\documentclass[11pt]{amsart}

\usepackage{amsmath, amscd, amssymb}
\usepackage[frame,cmtip,arrow,matrix,line,graph,curve]{xy}
\usepackage{graphpap, color}
\usepackage[mathscr]{eucal}
\usepackage{verbatim}

\usepackage{float}
\usepackage{booktabs}
\usepackage{multirow}
\usepackage{tabularx}
\usepackage{tikz-cd, tikz}
\usepackage{enumitem}

\usepackage{hyperref}
\newif\ifdraft
\draftfalse
\newif\ifcite
\citefalse
\ifdraft\ifcite\usepackage{showkeys}\else\usepackage[notcite,notref]{showkeys}\fi\fi

\usepackage{color}

\usepackage{marginnote}

\numberwithin{equation}{section}

\newtheorem{theorem}{Theorem}[section]
\newtheorem{corollary}[theorem]{Corollary}
\newtheorem{proposition}[theorem]{Proposition}
\newtheorem{lemma}[theorem]{Lemma}

\newtheorem{conjecture}[theorem]{Conjecture}

\theoremstyle{definition}
\newtheorem{definition}[theorem]{Definition}
\newtheorem{question}[theorem]{Question}
\newtheorem{example}[theorem]{Example}
\newtheorem{remark}[theorem]{Remark}

\newcommand{\Z}{\mathbb{Z}}

\newcommand{\C}{\mathbb{C}}

\newcommand{\A }{\mathbb{A}}
\newcommand{\PP}{\mathbb{P}}

\def\ZZ{\mathbb{Z}}

\newcommand\cB{{\mathcal B}}

\newcommand\cI{{\mathcal I}}

\newcommand\cO{{\mathcal O}}

\def\fD{\mathfrak{D}}

\def\fX{\mathfrak{X} }

\def\fb{\mathfrak{b}}

\def\fg{\mathfrak g}

\def\ft{\mathfrak t}

\def\tG{\widetilde{G}}

\def\tM{\widetilde{M}}

\def\hbar{\overline{h}}

\def\GL{\mathrm{GL} }

\DeclareMathOperator{\id}{id}
\DeclareMathOperator{\rank}{rank}

\DeclareMathOperator{\Hom}{Hom}

\def\dim{\mathrm{dim} }

\def\Bl{\mathrm{Bl}}

\def\Pic{\mathrm{Pic} }

\def\and{\quad{\rm and}\quad}
\def\lra{\longrightarrow }

\def\beq{\begin{equation}}
\def\eeq{\end{equation}}
\def\ben{\begin{enumerate}}
\def\een{\end{enumerate}}

\def\and{\quad\text{and}\quad}

\def\id{\mathrm{id}}

\def\a{\alpha}

\newcommand{\hooklongrightarrow}{\lhook\joinrel\longrightarrow}

\def\Spec{\mathrm{Spec}}

\def\rs{\mathrm{rs}}
\def\grs{\mathfrak{g}^{\mathrm{rs}}}
\def\pgrs{\PP(\fg)^{\rs}}
\def\Grs{G^{\mathrm{rs}}}

\def\s{\mathbf{s}}
\def\Y{Y_{w}(\s)}
\def\Ycell{Y_{w}^\circ(\s)}
\def\Yv{Y_{v}(\s)}
\def\Ye{Y_e(\s)}
\def\Yvcell{Y_{v}^\circ(\s)}
\def\bsY{Y_{\underline{w}}(\s)}
\def\bsYcell{Y_{\underline{w}}^\circ(\s)}
\def\bsYjc{Y_{\underline{w}_{\{j\}^c}}(\s)}
\def\bsYI{Y_{\underline{w}_I}(\s)}
\def\bsYv{Y_{\underline{v}}(\s)}

\def\YDL{Y_w^{\mathrm{DL}}}
\def\YDLcell{(Y_w^\circ)^{\mathrm{DL}}}
\def\bsYDL{Y_{\underline w}^{\mathrm{DL}}}
\def\DL{Deligne-Lusztig }

\newlength\cellsize  

\savebox2{\begin{picture}(33,33)\linethickness{0.6pt} \put(0,0){\line(1,0){33}}
\put(0,0){\line(0,1){33}}
\put(33,0){\line(0,1){33}}
\put(0,33){\line(1,0){33}}
\end{picture}}

\newcommand\cellify[1]{\def\thearg{#1}\def\nothing{}\ifx\thearg\nothing\vrule width0pt height\cellsize depth0pt\else\hbox to 0pt{\usebox2\hss}\fi \vbox to \cellsize{\vss\hbox to \cellsize{\hss$_{#1}$\hss}\vss}}

\newcommand\tableau[1]{\vtop{\let\\=\cr
\setlength\baselineskip{-1000pt}
\setlength\lineskiplimit{1000pt}
\setlength\lineskip{0pt}
\ialign{&\cellify{##}\cr#1\crcr}}}

 \title{Geometry of regular semisimple Lusztig varieties}
\date{July 7, 2026}

\author{Patrick Brosnan}
\address{Department of Mathematics\\
  University of Maryland\\
  College Park, MD USA}
\email{pbrosnan@math.umd.edu}

\author{Jaehyun Hong}
\address{Center for Complex Geometry, Institute for Basic Science (IBS), 55 Expo-ro, Yuseong-gu, Daejeon 34126, Korea}
\email{jhhong00@ibs.re.kr}

\author{Donggun Lee}
\address{June E Huh Center for Mathematical Challenges, Korea Institute for Advanced Study, 85 Hoegiro, Dongdaemun-gu, Seoul 02455, Republic of Korea}
\email{dglee@kias.re.kr}

\begin{document}

\begin{abstract}
Lusztig varieties are subvarieties in flag manifolds $G/B$ associated to an element $w$ in the Weyl group $W$ and an element $x$ in $G$, introduced in Lusztig's papers on character sheaves.
We study the geometry of these varieties when $x$ is regular semisimple.

In the first part, we 
establish that they are normal, Cohen-Macaulay, of pure expected dimension and have rational singularities. We then show that the cohomology of ample line bundles vanishes in positive degrees, in arbitrary characteristic. This extends to nef line bundles when the base field has characteristic zero or sufficiently large characteristic.
Along the way, we prove that Lusztig varieties are Frobenius split in positive characteristic and that their open cells are affine. 
We also prove that the open cells in Deligne-Lusztig varieties are affine, settling a question that has been open since the foundational paper of Deligne and Lusztig.

In the second part, we explore their relationship with regular semisimple Hessenberg varieties. 
Both varieties admit Tymoczko's dot action of $W$ on their (intersection) cohomology. We associate to each element $w$ in $W$ a Hessenberg space
using the tangent cone of the Schubert variety associated with $w$, and show that the cohomology of the associated regular semisimple Lusztig varieties and  Hessenberg varieties is isomorphic as graded $W$-representations when they are smooth.

This relationship extends to the level of varieties: we construct a flat
degeneration of regular semisimple Lusztig varieties to regular semisimple
Hessenberg varieties. In particular, this proves a conjecture of Abreu and
Nigro on the homeomorphism types of regular semisimple Lusztig varieties in
type $A$, and generalizes it to arbitrary Lie types.

 \end{abstract}

\maketitle
\tableofcontents

\section{Introduction}

In a series of papers, Lusztig developed a theory of characters of a reductive
algebraic group $G$ by using perverse sheaves on $G$
(\cite{Lus85,Lus85ii,Lus86iv,Lus86v}). To get appropriate perverse sheaves on
$G$ (called character sheaves), he considered a family of subvarieties  of the
full flag variety $ G/B$ parameterized by elements in $G$. This paper is about
the geometry of these subvarieties.  

To state more precisely, let $k$ be an algebraically closed field and let $G$
be a reductive algebraic group over $k$.  Let $B$ be a Borel subgroup
containing a maximal torus $T$, and let $W=N(T)/T$ be its Weyl group. Then, $G$
is the disjoint union 
\[G=\bigsqcup_{w \in W}BwB\] 
of double cosets. The image $X_w$  of the closure $\overline{BwB}$ by the projection $G \rightarrow \mathcal B:=G/B$ defines a subvariety of $\mathcal B$, called the \emph{Schubert variety} associated to $w \in W$.  

 For $w \in W$, 
 let $Y_w$ be the subvariety in $G \times \mathcal B$ consisting of pairs $(x, gB)$ 
 such that $g^{-1}xg \in \overline{BwB}$, and let $\pi_w:Y_w \subset G\times \mathcal B\rightarrow G$ be the projection (\cite[Section 12.1]{Lus85ii}).  
 The fiber of $\pi_w$ over $x \in G$, denoted by $Y_w(x)$, is called the \emph{Lusztig variety} associated to $w$ and $x$.

 In other words, the Lusztig variety $Y_w(x)$ is defined as the fiber product
 	\beq\label{10}
		\begin{tikzcd}
			Y_w(x)\arrow[r,hook]\arrow[d,hook']&\fX_w\arrow[d,hook']\\ \cB\arrow[r, hook,"\iota_x"]&\cB\times \cB
		\end{tikzcd} 
	\eeq
  where $\iota_x:=(\id, x)$ is the map sending $gB$ to $(gB,x gB)$ and $\mathfrak X_w:=G \times^B X_w$ is the quotient of $G\times X_w$ by the diagonal action of $B$. Note that $\fX_w$ admits
  a closed immersion
  into $G \times^B\mathcal B \cong \mathcal B \times \mathcal B$. We call $\fX_w$ the \emph{relative Schubert variety} (associated to $w$).
  From this description, one may expect that Lusztig varieties share many nice properties with Schubert varieties.

\medskip

In the first part of this paper we generalize the following results on Schubert varieties to (regular semisimple) Lusztig varieties.
\begin{enumerate}
	\item Schubert varieties are normal and Cohen-Macaulay, and have rational resolutions.
	\item The cohomology groups on Schubert varieties of the restriction of nef line bundles on the ambient flag varieties vanish in positive degrees.
\end{enumerate}

From now on, we assume that $\s \in G$ is regular semisimple, that is, its centralizer is a maximal torus. In this case, we call $\Y$ \emph{regular semisimple}. 

\smallskip

The statements in Theorem~\ref{22} below may be well-known to experts, but we
could not find either the statements or complete proofs in the literature. 
For the sake of completeness, we provide a complete proof in
Section~\ref{s:basic.properties}.\footnote{
In March 2026, we received an email from G.~Lusztig pointing out that 
his paper~\cite{lusztig-twopartitions} has results related to this paper.
In fact, the arguments in   \cite[Section~1.6]{lusztig-twopartitions} show that every local ring of a closed point
in a regular semisimple Lusztig variety $\Y$ is isomorphic to the 
local ring of a closed point in the Schubert variety $X_w$.
This implies the first statement of the following Theorem~\ref{22}
as well as the ``if" direction of the second statement.
}

\begin{theorem}\label{22}
A regular semisimple Lusztig variety $\Y$ is normal, Cohen-Macaulay, of pure
dimension $\ell(w)$, and has at worst rational singularities.
Furthermore, $\Y$ is smooth (resp. Gorenstein) if and only if $X_w$ is.
\end{theorem}

For each $\lambda\in X(T):=\Hom(T,\mathbb{G}_m)\cong \Z^{\rank(T)}$, the natural projection $B\to T$ induces a $B$-module structure on
$k_{-\lambda}$ 
and 
the associated line bundle
\[L_\lambda:=G\times^{B}k_{-\lambda}\]
 on $\cB$. 
We call $\lambda$ \emph{dominant} (resp. \emph{regular dominant}) if its
pairing with every simple coroot has a nonnegative (resp. positive) value with
respect to the standard pairing 
\[X(T)\times\Hom(\mathbb{G}_m,T)\to \Hom(\mathbb{G}_m,\mathbb{G}_m)\cong\Z.\] 
If this is the case, $L_\lambda$ is
globally generated (resp. ample), and vice versa.
 
The following theorem is proved in Sections~\ref{sect ch 0} and~\ref{s.cohomology}, depending on whether the characteristic of the base field $k$ is zero or positive.
\begin{theorem}[Corollaries~\ref{cor:vanishing} and \ref{cor:vanishing.Frob}]  \label{23} Let $k$ be of characteristic $p \geq 0$. 
For   any regular dominant $\lambda \in  X(T)$, 
\begin{enumerate}
 \item \label{24}   $H^i(\Y, L_{\lambda})=0$ for any $i >0$;
 \item \label{24'} If $\Yv\subset \Y$, then $H^0(\Y, L_{\lambda}) \rightarrow H^0(\Yv, L_{\lambda})$ is surjective.
 \end{enumerate}
 When $p$ is zero or sufficiently large, \eqref{24} holds also for $\lambda$ dominant. 

\end{theorem}
\begin{remark}
	In fact, we prove a stronger result: the above statements remain true when
	dominant (resp.\ regular dominant) $L_\lambda$ are replaced by 
	nef (resp.\ ample) line bundles (Theorems~\ref{thm:vanishing}
	and~\ref{thm:vanishing.Frob}). 
	On the other hand,
	Theorem~\ref{23}\eqref{24'}
	does not necessarily hold for $\lambda$ dominant. See
	Remark~\ref{rem:vanishing}.
\end{remark}

There are various proofs of this kind of vanishing theorem for the cohomology of line bundles on  Schubert varieties. Some of these arguments extend to Lusztig varieties, while others do not. 
The main difference between Schubert varieties and Lusztig varieties is that,
while Schubert varieties are stable under the action of of a Borel subgroup
$B\subset G$, in general, Lusztig varieties are not.  This leads
to
different natures
of their open cells  
\[X_w^{\circ}=BwB/B=X_w \backslash \partial X_w \and \Ycell=Y_w(\mathbf{s}) \backslash \partial Y_{w}(\mathbf{s})\] 
respectively, where $\Ycell$ and $\partial Y_{w}(\mathbf{s})$ are by definition the preimages of $\fX_w^\circ:=G\times^B X_w^\circ$ and $\partial\fX:=G\times^B\partial X_w$ respectively.
Every Schubert cell is by definition a single orbit of $B$, which forces it to
be an affine space, but Lusztig cells do not have such nice properties.

Theorem~\ref{23} is
proved using 
two different tools: the Kawamata-Viehweg vanishing theorem in
characteristic zero (Section~\ref{sect ch
0}), and Frobenius splitting in
positive characteristic
(Section~\ref{s.cohomology}).
To be precise, both tools are applied to \emph{Bott-Samelson resolutions} of $\Y$, rather than directly to $\Y$. 
 
The Bott-Samelson resolution $\bsY\to \Y$ is a resolution of singularities which
is an isomorphism over $\Ycell$ and whose preimage over $\partial\Y$, denoted
by $\partial \bsY$, is a simple normal crossing divisor. 
It is obtained by pulling back the Bott-Samelson resolution of $\fX_w$ along $\Y\hookrightarrow \fX_w$.

Applying the adjunction formula to the pullback diagram, 
we first
show that the boundary divisor $\partial Y_{\underline{w}}(\mathbf{s})$ is an anti-canonical divisor of  $\bsY$ (Proposition~\ref{26}). 
In positive characteristic, this immediately implies:
\begin{theorem}\label{27}
	When $\mathrm{char}(k)>0$, $\bsY$ and $\Y$ are Frobenius split.
\end{theorem}
Theorem~\ref{23} for regular dominant $\lambda$ then follows directly from this (Theorem~\ref{28}).
Moreover, by semicontinuity of cohomology in families, this also provides a proof 
in characteristic zero, without relying on the Kawamata-Viehweg vanishing theorem (Subsection~\ref{20}).

To apply the Kawamata-Viehweg vanishing theorem, or
when $\lambda$ is only dominant,  we need more work. 
 A key ingredient is   that the boundary divisor $\partial Y_{\underline{w}}(\mathbf{s})$ supports an ample divisor (Proposition \ref{prop:ample_boundary}).     
When the characteristic of $k$ is zero, this leads to Theorem~\ref{23} via a careful application of the Kawamata-Viehweg vanishing theorem (Theorem~\ref{thm:vanishing}).

When the characteristic of $k$ is positive and sufficiently large,
$\bsY$
is Frobenius $D$-split for the ample divisor $D$ supported on $\partial\bsY$, which implies the vanishing for dominant $\lambda$ (Theorem \ref{thm:vanishing.Frob}).
This also provides another proof 
for characteristic zero, by 
semicontinuity of cohomology. 

The arguments in our proofs of Theorem~\ref{23} are essentially parallel with those in \cite{Brion} for Schubert varieties in characteristic zero, and in \cite{RR85,R87} for Schubert varieties, in \cite{BI94} for spherical varieties   in positive characteristic.

\smallskip

As a by-product of the proof, we get that Lusztig cells $\Ycell$ are affine, as
they are isomorphic to the complements of ample divisors in $\bsY$. 

\begin{theorem}[Corollary~\ref{cor:cell.affine}]\label{29}
	Lusztig cells $\Ycell$ are affine.
\end{theorem}

This is not an obvious fact as $\fX_w^\circ$ is not affine in general (e.g.
$\fX_e^\circ=\cB)$.
It is worth noting that the same argument proves the affineness of
\emph{regular} Lusztig cells, that is, those $Y_w^\circ(x)$ defined with $x\in
G$ only regular (Remark~\ref{rem:regular.affine}), and of
\emph{Deligne-Lusztig} cells introduced in \cite{DL}, which we
explain below.\footnote{
  In \cite[Section~4.1]{lusztig-twopartitions}, Lusztig showed that $\Ycell$ is 
affine when $w$ has minimal length in its conjugacy class.}

When $\mathrm{char}(k)=p>0$, the \emph{Deligne-Lusztig variety} associated to $w\in W$ and a positive power $q=p^r$ of $p$ is defined as the fiber product 
\beq\label{13}
		\begin{tikzcd}
			\YDL\arrow[r,hook]\arrow[d,hook']&\fX_w\arrow[d,hook']\\ \cB\arrow[r, hook,"\iota_F"]&\cB\times \cB
		\end{tikzcd} 
\eeq
where $\iota_F:=(\id, F)$ and $F$ is the $q$-th Frobenius map. Let $\YDLcell:=\iota_F^{-1}(\fX_w^\circ)$ and call it the \emph{Deligne-Lusztig cell} associated to $w$ and $q$ (\cite[Definition~1.4]{DL}).
\begin{theorem}[Corollary~\ref{cor:DL}]
	Deligne-Lusztig cells $\YDLcell$ are affine.
\end{theorem}
To the best of our knowledge, the problem of affineness of Deligne-Lusztig 
cells has been open since~\cite{DL}.
However, many partial results were known.
Deligne and Lusztig proved that $\YDLcell$ is affine whenever $q$ is larger than the Coxeter number of $G$~\cite[Theorem~9.7]{DL}. 
The affineness was later established when $w$ has minimal length in its Frobenius conjugacy class.
We refer the reader to
\cite{DL,bonnafe-rouquier-affine,he-affine,orlik-rapoport,he-lusztig,harashita-affine}
for this and other results related to the problem of affineness of Deligne-Lusztig
cells. 
In~\S\ref{s-history}, we also give some background on the motivation for the problem 
and its history.

In the second part of the paper, we focus on relationships between Lusztig
varieties and Hessenberg varieties. Now we assume the base field to be $\C$, as
we will consider their singular cohomology and $W$-representations on them,
defined using the GKM theory, as explained below.

A subspace $H$ of the Lie algebra $\mathfrak g$ of $G$ is called a {\it linear Hessenberg space} if it is a $B$-submodule and contains the Lie algebra $\mathfrak b$ of $B$. 
The {\it Hessenberg variety} associated to such an $H$ and an element $s\in \fg$ is the subvariety
\[X_H(s):=\{gB \in\mathcal B :g^{-1}.s \in H\}\]
of the flag variety $\cB$, where $.$ denotes the adjoint action of $G$ on $\fg$.

This definition is due to de Mari, Procesi and Shayman \cite{demari-procesi-shayman}, who also proved that $X_H(s)$ is smooth of dimension $\dim ~H/\fb$ when $s\in \fg$ is regular semisimple, that is, its centralizer subgroup $C_G(s)$ in $G$ is a maximal torus.
In this case, $X_H(s)$ is equipped with an action of the maximal torus $C_G(s)$, as the restriction of the natural $G$-action on $\cB$.
The fixed locus of this action is  $X_H(s)^{C_G(s)}=\cB^{C_G(s)}\cong W$ consisting of $\lvert W\rvert$ points, inducing the Bia\l{}ynicki-Birula decomposition (\cite{Bia73}) of $X_H(s)$ by $\lvert W\rvert$ affine spaces.

As an easy consequence, $X_H(s)$ as a $C_G(s)$-variety is a \emph{GKM manifold}, which is by definition a smooth projective variety with an action of a torus that has only finitely many torus-invariant points and curves and that is equivariantly formal (for example, whose singular cohomology vanishes in odd degrees). One of the main consequences of the theory of GKM manifolds \cite{GKM} is that the cohomology of a GKM manifold is determined by its \emph{GKM graph}, which is the decorated graph consisting of the torus-fixed points as its vertices and the torus-invariant curves as its edges, with decorations on the edges by the torus-weights on the corresponding torus-invariant curves.

Later in \cite{tymoczko}, by applying the GKM theory to the Hessenberg variety $X_H(s)$, Tymoczko introduced an action of the Weyl group $W$ on its singular cohomology $H^*(X_H(s))$, known as the \emph{dot action}. This action, particularly in the case where $G=\GL_n(\C)$ and $W=S_n$, was found to have intriguing connections to other fields such as representation theory and combinatorics. These connections were explored through various conjectures, proofs, and counterexamples in works including \cite{AN2,brosnan-chow,CHSS,guay-paquet,haiman,kiem-lee,precup-sommers,shareshian-wachs}. 

Let
$X_H^{\rs}\to \grs$ denote the family whose fibers are regular semisimple Hessenberg varieties 
$X_H(s)$ over regular semisimple elements $s\in \grs\subset \fg$.
This induces a monodromy action of $\pi_1(\grs, s)$ on $H^*(X_H(s))$.
It was shown in \cite{brosnan-chow} in type $A$, and later generalized in \cite{balibnu-crooks} to other Lie types, that 
the monodromy action factors through the dot action and, in fact, recovers it.

A parallel story exists for regular semisimple Lusztig varieties. Considering the family $\pi_w:Y_w\to G$ of Lusztig varieties, it was proved in \cite{AN} that the monodromy action of $\pi_1(\Grs,\s)$ on the intersection cohomology $IH^*(\Y)$ of a regular semisimple Lusztig variety $\Y$ descends to an action of $W$. 
Moreover, when 
$\Y$ is smooth,  
this action is equivalent to the dot action, which can be defined via GKM theory, as $\Y$ is also a GKM manifold.

\smallskip
A comparison of these two varieties in the case of type $A$  was done quite successfully by Abreu and Nigro in \cite{AN2}: 
When $G=\GL_n(\C)$, the group $G$ naturally embeds into  
its Lie algebra $\mathfrak g=\mathfrak{gl}_n(\C)$ by identifying them with the sets of (invertible) $n\times n$ matrices. So, we may choose $\s =s$ in $\Grs \subset \mathfrak \grs$ as the same (regular semisimple) matrix. 
Then, for each linear Hessenberg space $H$, there is a 
unique \emph{codominant} permutation
$w \in S_n$ such that 
\beq\label{4}X_H (s)=\Y.\eeq
The notion of codominant permutations is first introduced in \cite{haiman}. They are precisely the permutations that avoid  the pattern 312 while smooth permutations are those avoiding the patterns 4231 and 3412. In particular, codominant permutations form a special class of smooth permutations. For further details, see Section \ref{ss:Hess space}. 

Abreu and Nigro noticed that there is a natural map
\beq\label{6}\{\text{smooth elements in }S_n\}\lra \{\text{linear Hessenberg subspaces in }\mathfrak{gl}_n\}\eeq
defined in a purely combinatorial manner,
which restricts to a bijection on the subset of codominant permutations 
identifying $w$ and $H$ in \eqref{4}.
This shows that for a smooth permutation $w'$, there exists a unique codominant $w$ such that $w$ and $w'$ have the same image $H$ under \eqref{6}. Then, they proved that 
$X_H(s)=\Y$ and $Y_{w'}(\s)$ have the same GKM graphs, in particular,
\beq\label{7}H^*(X_H(s))=H^*(Y_w(\s))\cong H^*(Y_{w'}(\s))\eeq
as  graded $S_n$-representations (\cite[Theorem~1.6]{AN2}). 
This means that the map \eqref{6} identifies smooth permutations and linear Hessenberg subspaces that give rise to the same $S_n$-representations \eqref{7}. 

As a strengthening of this, Abreu and Nigro conjectured the following. 
\begin{conjecture}[Conjecture~3.9 in \cite{AN2}]\label{conj}
	Two regular semisimple Lusztig varieties $\Y$ and $Y_{w'}(\s)$ in \eqref{7} are homeomorphic. 
\end{conjecture}
In Section~\ref{sec:degen}, which is the last section of this paper, we extend the map \eqref{6} and the isomorphism \eqref{7} to an arbitrary Lie type (Corollary~\ref{cor:Wrep}). We also extend Conjecture~\ref{conj} to an arbitrary Lie type, and provide its proof (Corollaries~\ref{c.diffeo} or~\ref{cor:diffeo}).
These extensions suggest and are based on how we should view these two varieties with different natures in a uniform way.

Since there is no natural inclusion of $G$ into $\fg$ in general, it is not natural to expect an equality \eqref{4} as in the case of type $A$. Instead of trying this,
we construct a flat family of regular semisimple  Lusztig varieties that degenerate to regular semisimple Hessenberg varieties (Theorem~\ref{thm:univ.fam}). The construction shows that we may regard Hessenberg varieties as linearized versions of Lusztig varieties. 
While doing this, we generalize the definition  of linear Hessenberg spaces (Definition~\ref{d.Hess}) to make singular elements $w\in W$ fit into this picture in a uniform way. 

\smallskip

For simplicity, we continue to restrict ourselves to the case when $w \in W$ is smooth, i.e.\ when $X_w$ is smooth. Consider the map 
\beq \label{e.sm} \{w\in W \text{ smooth}\}\lra \{\text{linear Hessenberg spaces in }\fg\}\eeq
which sends $w$ to the unique $ H_w\subset \fg$ containing $\fb$ such that $H_w/\fb$ is the tangent space of $X_w$ at the identity, 
as a subspace in the tangent space $\fg/\fb$ of $\cB$ at the identity.
This generalizes \eqref{6} to an arbitrary type (Corollary~\ref{cor:corres.A}). 

The following results generalize \eqref{7}. 

Let $T^{\rs}:=T\cap \Grs$ and $\ft^{\rs}:=\ft\cap \grs$, where $\ft$ denotes the Lie algebra of $T$. Note that every element in $T^{\rs}$ and $\ft^{\rs}$ has the same centralizer subgroup $T$.
\begin{theorem}[Theorem~\ref{thm GKM graph}]\label{11}
	Let $\s\in T^{\rs}$ and $s\in \ft^{\rs}$, and let $w\in W$ be smooth.
	Then $\Y$ and $X_{H_w}(s)$ have the same GKM graphs. In particular, if $w,w'\in W$ are smooth and have the same image $H$ under \eqref{e.sm}, then
	\[H^*(\Y)\cong H^*(X_{H}(s))\cong H^*(Y_{w'}(\s))\]
	as graded $W$-representations.
\end{theorem}

Since the $W$-representation on $H^*(\Y)$ does not depend on the choice of $\s\in \Grs$ (see \eqref{eq:conjugate}), Theorem~\ref{11} implies the following. 
\begin{corollary}\label{cor:Wrep}
	Suppose that $w, w'\in W$ are smooth with $H_w=H_{w'}$. Then, 
	\[H^*(\Y)\cong H^*(Y_{w'}(\s'))\]
	for any  $\s$ and $\s'\in \Grs$, as graded $W$-representations.
\end{corollary}

Theorem~\ref{11} extends to a
statement about the degeneration of the family of Lusztig varieties associated to a
smooth element $w$ in $W$ to the corresponding family of Hessenberg varieties.
To explain it, we first introduce some necessary notation.
\begin{itemize}
	\item Let $Y_w^{\rs}\to \grs$ denote the restriction of $Y_w\to \fg$ over $\grs$. This is a smooth family of Lusztig varieties, whose fiber over $\s \in \grs$ is $\Y$.
	\item Define $\pgrs$ as the image of $\grs-0$ in $\PP(\fg)$. Let 
$X_{\PP(H)}^{\rs}\to \pgrs$ denote  the smooth family of Hessenberg varieties,  whose fiber over $[s]$ is
$X_{H}(s)$. This is well-defined since $X_{H}(s)=X_{H}(cs)$ for $c\in\mathbb{G}_m$.
	\item Let $\widetilde G$ denote the blowup of $G$ at the identity element $e$, and let
$\widehat G$ denote the complement in $\widetilde G$ of the proper transform of
$G - G^{\rs}$. 
Note that $\widehat G$ is smooth and connected. Set-theoretically, it is the disjoint union $\widehat G=\Grs\sqcup \pgrs$ with $\Grs$ open and $\pgrs$ closed.
\end{itemize}

\begin{theorem}[Theorem~\ref{thm:univ.fam}]\label{12}
  Let $w\in W$ be smooth. Then, there exists a smooth projective morphism $\widehat Y_w \to \widehat G$ 
	which satisfies the following. 
	\begin{enumerate}
		\item The family restricts to $Y_w^{\rs} \to \Grs$ over $\Grs$.
		\item The family restricts to $X_{\PP(H_w)}^{\rs}\to \pgrs$ over $\pgrs$. 
	\end{enumerate}
	In short, we have fiber diagrams 
	\[\begin{tikzcd}
			Y_w^{\rs}\arrow[r,hook]\arrow[d]&\widehat Y_w \arrow[d]&X_{\PP(H_w)}^{\rs}
			\arrow[l,hook']\arrow[d]\\ 
			\Grs\arrow[r,hook]& \widehat G&\pgrs \arrow[l,hook']
		\end{tikzcd}\]
	where the inclusions on the left are open immersions, while those on the right correspond to their closed complements. 
\end{theorem}
The variety $\widehat Y_w$ is constructed by pulling back the family of
diagrams \eqref{10} over $G$ along $\widehat G\to
G$ and applying a deformation-to-normal-cone type argument
in \cite{fulton-intersection} to the pulled-back family.

Applying the Ehresmann theorem to $\widehat Y_w \to \widehat G$ and $\widehat Y_{w'} \to \widehat G$ constructed above, and to $X_{\PP(H)}^{\rs}\to \pgrs$ when $H=H_w=H_{w'}$, we obtain the following.
\begin{corollary}[Corollary~\ref{cor:diffeo}]\label{c.diffeo}
	Let $w,w'\in W$ be smooth, and let $\s\in \Grs$. 
	If $X_w$ and $X_{w'}$ have the same tangent space at the identity, viewed as subspaces of the tangent space of $\cB$, 
	then $\Y$ and $Y_{w'}(\s)$ are diffeomorphic.
\end{corollary}
This proves Conjecture~\ref{conj}, originally stated for $G=\GL_n(\C)$ and $W=S_n$, and generalizes it to arbitrary Lie types.

\smallskip

The family $\widehat Y_w\to \widehat G$ may be of independent interest. For instance, it could have implications for the moduli spaces of regular semisimple Lusztig varieties or Hessenberg varieties (cf.~\cite{BEHLLMS}).

\smallskip
Another application of Theorem~\ref{12} is the Weyl group symmetry of weight multiplicities for the $C_G(\s)$-representation 
\[V_w(\lambda):=H^0(\Y,L_\lambda),\] 
where $\lambda$ is dominant and $w$ is smooth. This follows from the well-known symmetry for $H^0(\cB,L_\lambda)$, since $V_w(\lambda)$ can be expressed as a linear combination of $H^0(\cB,L_\mu)$ as $C_G(\s)$-representations. See Proposition~\ref{prop:rep} and below.

\begin{corollary}\label{cor:weight.multi} Let $w\in W$ be smooth and let $\s\in \Grs$. Let $\lambda$ be a dominant integral weight, and let $\chi$ be a character of $C_G(\s)$. Then $V_w(\lambda)$ has weight spaces of the same dimension for $\chi$ and $\sigma(\chi)$ for any $\sigma\in W$.
\end{corollary}
This leads to the following natural question.
\begin{question}
Does the Weyl group symmetry of weight multiplicities of $V_w(\lambda)$
continue to hold when $w$ is singular?
\end{question}

We conclude with one further remark.
Thanks to the identification \eqref{7}, for $G=\GL_n(\C)$, Theorem~\ref{23} can be interpreted as a result for $X_{H_w}(s)$.

\begin{corollary}[Corollary~\ref{cor:vanishing.hess}] \label{c.vanishing A} 
Let $G=\GL_n(\C)$ and let $H, H'\subset \fg$ be linear Hessenberg spaces.
Then the following hold:
\begin{enumerate}
 \item $H^i(X_H(s), L_{\lambda})=0$ for any $i >0$ and dominant $\lambda$;
\item if $H\subset H'$, then the restriction map $H^0(X_{H'}(s), L_{\lambda}) \rightarrow H^0(X_H(s), L_{\lambda})$ is surjective for regular  dominant $\lambda$.
 \end{enumerate}
\end{corollary}

This generalizes the result of Abe, Fujita and Zeng \cite{AFZ} for weak Fano Hessenberg varieties in type $A$ (see Remark~\ref{rem:AFZ}).
\\

The paper is organized as follows.  In Section~\ref{s:basic.properties}, we collect basic facts on Schubert varieties, Bott-Samelson resolutions and their relative counterparts, and extend those results to Lusztig varieties (Theorem \ref{22}). 

In Section~\ref{sect ch 0}, we first prove that the boundary divisor $\partial \bsY$ of the Bott-Samelson resolution of $\Y$ is a simple normal crossing anti-canonical divisor. Then, we review known results on line bundles on the Bott-Samelson resolutions of $X_w$ and $\fX_w$,  and using these, we show that the divisor $\partial \bsY$  supports an ample divisor. In particular, this implies that Lusztig cells are affine. We then apply a version of the Kawamata-Viehweg vanishing theorem, stated in Theorem~\ref{thm:kawamata-viehweg}, to prove Theorem~\ref{23} when $\mathrm{char}~k=0$. 

In Section~\ref{s.cohomology}, we first explain how to use Frobenius splitting methods in general in the first subsection, and apply them to Lusztig varieties in the second subsection to deduce Theorems~\ref{23} and \ref{27} when $\mathrm{char}~k>0$. In the last subsection, we also explain how Theorem~\ref{23} in characteristic zero can be recovered from the result in positive characteristic.

In Section~\ref{sec:degen}, we generalize the usual Hessenberg varieties to \emph{Hessenberg schemes} in Subsection~\ref{ss:Hess space}, in order to include flat limits of $\Y$ with singular $w$ within the framework of Hessenberg schemes. In Subsection~\ref{ss:GKM}, we recall the comparison results \eqref{4}--\eqref{7}  by Abreu and Nigro  in type $A$, and extend these to an arbitrary type, proving Theorem~\ref{11}. In Subsection~\ref{ss:degeneration}, we construct the flat family $\widehat Y_w\to \widehat G$ in Theorem~\ref{12} for any $w\in W$, and consequently, prove Corollary~\ref{c.diffeo}.  
In Subsection~\ref{ss:cohomology.revisited}, we provide additional results on the cohomology of line bundles, obtained by combining Theorem~\ref{23} with the relationships between Lusztig varieties and Hessenberg varieties.
As an application, we obtain a proof of Corollary~\ref{cor:weight.multi}.

\medskip

\noindent{\bf Acknowledgments.}
We would like to thank Michel Brion for
valuable discussions, especially on Frobenius splitting methods, and Eunjeong
Lee for bringing \cite{BS24} to our attention.    
We also thank the anonymous referees for their careful reading and helpful
comments, which led to several improvements of the paper, including the
addition of Corollary~\ref{cor:weight.multi}.
We thank George Lusztig for pointing out that his
paper~\cite{lusztig-twopartitions} contains results related to ours, 
Michael Rapoport for suggesting us to expand the discussion of the history of
the affineness problem for Deligne-Lusztig varieties and Tom Haines for helpful comments 
about this history.
 
Hong was supported by the Institute for Basic Science IBS-R032-D1. Lee was
supported by the Institute for Basic Science IBS-R032-D1, the National
Research Foundation of Korea (NRF) grant funded by the Korea government (MSIT)
2021R1A2C1093787, and the KIAS Individual Grant (HP109201).

\section{Singularities and resolutions}\label{s:basic.properties}

In this section, we first review well-known results on singularities and (Bott-Samelson) resolutions of Schubert varieties. Then, we extend them to regular semisimple Lusztig varieties. In particular, we prove Theorem~\ref{22}.
\subsection{Schubert varieties} We review basic properties of Schubert varieties. These will extend to Lusztig varieties in the next subsections.

Let $k$ be an algebraically closed field and let $G$ be a semisimple algebraic group over $k$. Let $T$ be a maximal torus with the normalizer $N(T)$, and let $B$ be a Borel subgroup  containing $T$. Let $W=N(T)/T$ be the Weyl group. Let $\cB:=G/B$ be the flag variety. 

We denote the \emph{Schubert variety} (resp. \emph{cell}) associated to $w\in W$ by 
	\[X_w:=\overline{B w B}/B \quad  \text{(resp. }X_w^\circ:=B {w}B/B)\]
	where, with an abuse of notation, we use $w$ for any of its lifts in $N(T)$.

\medskip

These varieties are invariant subvarieties in $\cB$ under the natural $B$-action.
Given a variety $X$ with a (left) $B$-action, one can associate a $G$-variety \[G\times^B X:=(G\times X)/B\]
that is the quotient
by the diagonal $B$-action $((g,x),b)\mapsto (gb,b^{-1}x)$ for $(g,x)\in G\times X$ and $b\in B$.
This $B$-action commutes with the $G$-action given as the left multiplication on $G$ and the trivial action on $X$,
inducing a $G$-action on $G\times^B X$.  Moreover, by the first projection, it is naturally  a $G$-equivariant fibration over $\cB$ with fibers isomorphic to $X$.
	
When $X$ is a flag variety $\cB$, there is an isomorphism
\beq\label{eq:isom.relative}G\times^B\cB\xrightarrow{~\cong~} \cB\times \cB, \quad [g,x]\mapsto (gB,gx),\eeq
where $[g,x]$ denotes the $B$-equivalence class of $(g,x)$. This is $G$-equivariant with respect to the diagonal $G$-action on $\cB\times \cB$.

\begin{definition}
	Define the \emph{relative Schubert variety} (resp. \emph{cell}) by
	\[\fX_w:=G\times^BX_w \quad \text{(resp.~}~ \fX_w^\circ:=G\times^BX_w^\circ).\]
Under the isomorphism \eqref{eq:isom.relative}, these varieties  are subvarieties \[\begin{split}
		\fX_w&=\{(gB,x)\in \cB\times \cB :g^{-1}x\in X_w\},\\
		\fX_w^\circ&=\{(gB,x)\in \cB\times \cB :g^{-1}x\in X_w^\circ\}
	\end{split}
	\]
	in $\cB\times \cB$,  
	respectively.

\end{definition}

The Bruhat decomposition $G=\bigsqcup_{w\in W}B\dot w B$ of $G$ induces decompositions
\[\cB=\bigsqcup_{w\in W}X_w^\circ \and \cB\times \cB =\bigsqcup_{w\in W}\fX_w^\circ\]
of $\cB$ and $\cB\times \cB$ by the (relative) Schubert cells. These are precisely the decompositions by $B$-orbits and $G$-orbits, respectively. Moreover, $X_w^\circ\cong\A^{\ell(w)}$ where $\ell(w)$ denotes the \emph{length} of $w$, the minimal number $\ell$ such that $w=s_{i_1}\cdots s_{i_\ell}$ with simple reflections $s_{i_j}$.
These induce decompositions
\[X_w=\bigsqcup_{v\leq w}X_v^\circ \and \fX_w=\bigsqcup_{v\leq w}\fX_v^\circ\]
of $X_w$ and $\fX_w$, where $\leq $ denotes the Bruhat order on $W$.

The following  basic properties are well-known for Schubert varieties.
\begin{proposition}\label{9}
Let $w\in W$. Then,  
$X_w$  and $\fX_w$ are irreducible of dimensions $\ell(w)$ and $\ell(w)+\dim~\cB$, respectively. Moreover, they are
normal and Cohen-Macaulay with rational singularities. \end{proposition}
We refer the reader to \cite{Brion} for the proof. Note that $X_w$ and $\fX_w$ share the same local properties due to the fibration $\fX_w\to \cB$ with fibers $X_w$.

\smallskip

\subsection{Bott-Samelson resolutions} 
There are well-studied resolutions of the singularities of Schubert varieties, called the \emph{Bott-Samelson resolutions}. Their construction depends on the choice of a reduced word of a given $w\in W$.  
A \emph{word} of $w\in W$ is a sequence $\underline{w}=(s_{i_1},\cdots ,s_{i_\ell})$ of simple reflections such that 
$w=s_{i_1}\cdots s_{i_\ell}$. We say that $\underline w$ is \emph{reduced} if $\ell=\ell(w)$.

For each simple reflection $s_i \in W$, let $P_i \subset G$ denote the corresponding minimal (standard) parabolic subgroup containing $B$.

\medskip

The \emph{Bott-Samelson variety} associated to $\underline w$ is the quotient
\[X_{\underline{w}}:=P_{i_1}\times^B\cdots \times^B P_{i_\ell}/B\]
of the product   $P_{i_1} \times \dots \times P_{i_{\ell}}$  by the $B^{\ell}$-action defined by 
\[(p_1,\cdots, p_\ell).(b_1,\cdots,b_{\ell})=(p_1 b_1^{},b_1^{-1}p_2 b_2,\cdots, b_{\ell-1}^{-1} p_\ell b_\ell)\]
for $p_j\in P_{i_j}$ and $b_j\in B$, where $1\leq  j \leq \ell$. Its open cell $X_{\underline w}^\circ$ is defined in the same manner by replacing $P_{i_j}$ by $B s_{i_j} B$.
We denote its boundary by $\partial X_{\underline w}:=X_{\underline w}-X_{\underline w}^\circ$.

Consider the following two natural maps:
the multiplication map 
\beq\label{35}X_{\underline{w}}\lra \cB\eeq
sending the $B^{\ell}$-equivalence class $[p_1,\cdots,p_\ell]$
to $p_1\cdots p_\ell B$, and the inclusion
\beq\label{33}X_{\underline w_I}\hooklongrightarrow X_{\underline w} \eeq
associated to any subset $I\subset\{1,\cdots, \ell\}$, defined by setting $p_j=1$ for $j\notin I$.
Here, $\underline w_I$ is the associated subword of $\underline w$ given as
\[\underline w_I:=(s_{i_{j_1}},\cdots, s_{i_{j_m}}) ~\text{ for }I=\{j_1,\cdots, j_m\}.\] 
When $I=\{j\}^c$ for some $1\leq j\leq \ell$, the image of \eqref{33} is a smooth divisor in $X_{\underline w}$, denoted by $D_j$. Then, one can easily see that $X_{\underline w_I}\cong \bigcap_{j\notin I}D_j$.

Using these maps, one can easily prove the following. \begin{lemma}\label{36}
	Let $\underline w=(s_{i_1},\cdots, s_{i_\ell})$ be a word. Then, the following hold.
	\begin{enumerate}
		\item When $\underline w $ is a reduced word of $w\in W$, the map \eqref{35} has the image $X_w$, and it restricts to an isomorphism $X_{\underline w}^\circ \xrightarrow{\cong}X_w^\circ$ over $X_w^\circ$.
\item $\partial X_{\underline w}=\bigcup_{j=1}^\ell D_j$, and this is a simple normal crossing divisor.
		\item For any subword $\underline w_I$ of $\underline w$ with $I\subset \{1,\cdots,\ell\}$, the image of $X_{\underline w_I}$ under \eqref{33} is the transverse intersection of $D_j$ with $j\notin I$.
	\end{enumerate}
\end{lemma}
Due to (1), we call $X_{\underline w}$ the \emph{Bott-Samelson resolution} of $X_w$ associated to $\underline w$ when $\underline w$ is a reduced word of $w$. 

\smallskip

Note that for each $k$, the multiplication map
\beq \label{38}P_{i_k}\times^B (P_{i_{k+1}}\cdots P_{i_\ell})/B \lra (P_{i_k}\cdots P_{i_\ell})/B\eeq
factors through the projection $\PP^1\times (P_{i_k}\cdots P_{i_\ell})/B \longrightarrow (P_{i_k}\cdots P_{i_\ell})/B $, where the immersion $P_{i_k}\times^B (P_{i_{k+1}}\cdots P_{i_\ell})/B \lra P_{i_k}/B\times (P_{i_k}\cdots P_{i_\ell})/B$ is given by $[g,x] \mapsto (gB,gx)$ and $   P_{i_k}/B$ is identified with $\mathbb P^1$.

\begin{lemma}\label{39} {\rm (\cite[\S2.1]{Brion})}
	Let $\underline w$ be a reduced word of $w\in W$. 
Then, the map $X_{\underline w}\to X_w$ in \eqref{35} factorizes into the closed immersions and the $\PP^1$-bundle maps as follows:
	$X_{k+1}\hookrightarrow \hat X_k \to X_k$
induced by \eqref{38} for $1\leq k< \ell$,
	where \beq\label{54}\begin{split}
		&X_k:= P_{i_1}\times^B\cdots \times^B P_{i_{k-1}}\times^B (P_{i_k}\cdots P_{i_{\ell}})/B\\
		&\hat X_k:=P_{i_1}\times^B\cdots \times^B P_{i_{k-1}}\times^B \Big(P_{i_k}/B\times (P_{i_k}\cdots P_{i_{\ell}})/B\Big)
	\end{split}
	\eeq
so that $X_\ell=X_{\underline w}$ and $X_1=X_w$.
\end{lemma}

There are analogous constructions for relative Schubert varieties. Let $\underline w$ be any word as above.
Define 
\[\fX_{\underline w}:=G\times^B X_{\underline w} \and \fX_{\underline w}^\circ :=G\times^BX_{\underline w}^\circ\]
using the natural $B$-actions on $X_{\underline w}$ and $X_{\underline w}^\circ$.
Then, the map \eqref{35} induces \beq\label{37}\fX_{\underline w}\lra  G\times^B\cB\cong \cB\times \cB\eeq For $1 \leq j \leq \ell$, let
\[\fD_j:=G\times^BD_j\cong \fX_{\underline w_{\{j\}^c}}.\]
The following are immediate from Lemmas~\ref{36} and \ref{39},  respectively.
\begin{lemma}\label{40}
Let $\underline w$ be a word as above. 
Then, the following hold. 
\begin{enumerate}
	\item When $\underline w$ is a reduced word of $w\in W$, the map \eqref{37} has the image $\fX_w$, and it restricts to an isomorphism $\fX_{\underline w}^\circ \xrightarrow{\cong}\fX_w^\circ$ over $\fX_w^\circ$.
\item $\partial \fX_{\underline w}:=\fX_{\underline w}-\fX_{\underline w}^\circ$ is a simple normal crossing divisor $\bigcup_{j=1}^\ell \fD_j$. 
	\item More precisely, for $I\subset \{1,\cdots,\ell\}$, the image of $\fX_{\underline w_I}\hookrightarrow \fX_{\underline w}$
	induced from \eqref{33} is the transverse intersection of $\fD_j$ with $j\notin I$.
\end{enumerate}

\end{lemma}

Due to (1), we call $\fX_{\underline w}$ the \emph{Bott-Samelson resolution} of $\fX_w$ associated to $\underline w$ when $\underline w$ is a reduced word of $w$.

\begin{lemma}\label{19}
	Let $\underline w$ be a reduced word of $w\in W$.
Then, the map $\fX_{\underline w}\to \fX_w$ in \eqref{37} factors into the maps 
\beq\label{56}\fX_{k+1}\hooklongrightarrow \hat \fX_{k}\lra \fX_k\eeq
	where $\fX_k:=G\times^BX_k$ and $\hat \fX_k:=G\times^B \hat X_k$ for $1\leq k\leq \ell$. The first map in \eqref{56} is a closed immersion, and the second map is a $\PP^1$-bundle map. \end{lemma}

In the next subsections, we will extend all these properties in Proposition~\ref{9} and Lemmas~\ref{40} and~\ref{19} to regular semisimple Lusztig varieties.

\subsection{Lusztig varieties}
We recall the definition of a Lusztig variety, which is a fiber of the map $\pi_w:Y_w\to G$ introduced and studied in \cite{Lus85}. 
\begin{definition}
	Let $w\in W$ and $\s \in G$. The \emph{Lusztig variety} associated to the pair $(w,\s)$ is defined to be the fiber product
	\beq\label{3}
		\begin{tikzcd}
			\Y\arrow[r,hook,"\iota"]\arrow[d,hook']&\fX_w\arrow[d,hook']\\ \cB\arrow[r, hook,"\iota_\s"]&\cB\times \cB
		\end{tikzcd} 
	\eeq
where
$\iota_\s:=(\id, \s)$ is the map sending $gB$ to $(gB,\s gB)$. The \emph{Lusztig cell} is defined similarly by replacing $\fX_w$ by $\fX_w^\circ$, and denoted by  $\Ycell$. As we will see, $\Y$ and $\Ycell$ are reduced. Hence, equivalent definitions are
\[\begin{split}
		&\Y:=\{gB\in \cB: g^{-1}\s g \in \overline{B wB}\}\\
		\text{(resp. }~&\Ycell:=\{gB\in \cB: g^{-1}\s g \in {B wB}\})
	\end{split}\]
	as subvarieties in $\cB$.
\end{definition}
An element $\s\in G$ is \emph{regular semisimple} if its centralizer $C_G(\s)$ 
in $G$ is a maximal torus. Regular semisimple elements form a Zariski open subset in $G$. We denote this subset by $\Grs$.

We will always assume $\s\in \Grs$ unless otherwise mentioned. In this case, we call $\Y$ a \emph{regular semisimple} Lusztig variety (associated to $(w,\s)$).

\bigskip
The Bruhat decomposition $\fX_w=\bigsqcup_{v\leq w}\fX_v^\circ$ of $\fX_w$ induces a decomposition
\[\Y=\bigsqcup_{v\leq w}\Yvcell\]by Lusztig cells. 

\begin{remark}
	(1) 
	Unlike $X_w$ and $\fX_w$, the Lusztig variety $\Y$ is not necessarily connected. 
	For example, if $w$ is the longest element in a standard parabolic subgroup $P\subset G$, then $\overline{BwB}=P$. Therefore $\Y$ is the disjoint union of $\lvert W/W_P\rvert$ copies of $P/B$, where $W_P\subset W$ denotes the Weyl subgroup of $P$. In particular, when $w=e$, we have $\Ye=\{wB:w\in W\}\cong W$.

	(2)
	Unlike $X_w^\circ$, the Lusztig cell $\Ycell$ is not necessarily an affine space. For example, when $\ell(w)=1$, 
	$\Ycell$ is the disjoint union of $\mathbb{G}_m$.

\end{remark}

We first recall a result (due to Lusztig) that these cells are smooth. 
\begin{lemma}[Lemma~1.1 in \cite{Lus} and Remark below it] \label{2'}
	Let $\s\in \Grs$. The image of $\iota_\s$ intersects any $G$-orbit in $\cB\times \cB$ 
	(i.e., the $\fX_w^\circ$)
	transversely. 
	In particular, $\Ycell$ is smooth of pure dimension $\ell(w)$ if it is nonempty. 
\end{lemma}

From this, one can easily deduce the following.\footnote{In the regular case, the nonemptiness of $Y_w^\circ(s)$ is known and goes back to Springer, while the pure dimensionality is established in \cite{He-La} (see Remark~\ref{rem:regular}).
For the convenience of the reader, we include a direct proof in the regular semisimple case, using Lemma~\ref{2'}.}
\begin{proposition}\label{15c} Let $w\in W$ and $\s\in \Grs$. Then, $\Ycell$ is nonempty with  
	$\Y=\overline{\Ycell}$, and it has pure dimension $\ell(w)$.	
\end{proposition}
\begin{proof}
	Let $\Delta_\s$ denote the image of $\iota_\s$. Then, \[\Y\cong \Delta_\s\cap \fX_w=\bigsqcup_{v\leq w}\Delta_\s \cap \fX_v^\circ,\]
	and $\Y$ is nonempty as it contains $\Ye\cong W$. By the left-hand side, its irreducible components have dimension at least $\ell(w)=\dim ~\fX_w- \mathrm{codim}~\iota_\s$.
	Therefore, by the right-hand side, $\Ycell\cong \Delta_\s\cap \fX_w^\circ$ is nonempty and dense, as all the other strata $\Delta_\s\cap \fX_v^\circ$ have dimension $\ell(v)<\ell(w)=\dim \Y$.  
\end{proof}

Using these, one can deduce local properties for $\Y$ from those of $\fX_w$. Let $\mathrm{Sing}(-)$ denote the singular locus of $(-)$. Then, $\mathrm{Sing}(\fX_w)$ is $G$-invariant.
\begin{proposition}\label{5'}
	Let $w\in W$ and  $\s \in \Grs$. 
	A point $y \in \Y$ is a smooth point if and only if $\iota_\s(y)\in \fX_w$ is a smooth point. In other words,
	\beq\label{5}\mathrm{Sing}(\Y)=\iota_\s^{-1}(\mathrm{Sing}(\fX_w)).\eeq
More explicitly, if $S_w\subset W$ is a subset  with $\mathrm{Sing}(\fX_w)=\bigsqcup_{v\in S_w}\fX_v^\circ$, then
\[\mathrm{Sing}(\Y)=\bigsqcup_{v\in S_w}\Yvcell.\]
\end{proposition}
\begin{proof}
	Let $y\in \Y$. Then, $y\in \Yvcell$ for some $v\leq w$. Consider the commutative diagram of Zariski tangent spaces at $y$ or $\iota_\s(y)$
	\[\begin{tikzcd}
		0\arrow[r] & T_{\Yvcell,y}\arrow[r]\arrow[d,hook'] & T_{\cB,y}\oplus T_{\fX_v^\circ,\iota_\s(y)}\arrow[r]\arrow[d,hook'] & T_{\cB\times\cB,\iota_\s(y)}\arrow[r]\arrow[d,equal]&0 \\ 0\arrow[r]&T_{\Y,y}\arrow[r]&T_{\cB,y}\oplus T_{\fX_w,\iota_\s(y)}\arrow[r]&T_{\cB\times \cB,\iota_\s(y)}&
	\end{tikzcd}\] 
where the rows are induced from the isomorphisms $\Yvcell \cong \Delta_\s\cap
     \fX_v^\circ$ and $\Y\cong \Delta_\s\cap\fX_w$ respectively.
     
     Both rows are exact, the first row being exact by Lemma~\ref{2'}.
But, therefore, since the last column is an equality, the last map in
    the second row is onto. 
    Thus we have \[\dim~T_{\fX_w,\iota_\s(y)}-\dim~T_{\Y,y}=\dim~\cB=\dim~\fX_w-\dim~\Y\] by Proposition~\ref{15c}.
Therefore, $y\in \Y$ is smooth if and only if $\iota_\s(y)\in \fX_w$ is smooth.
The last assertion is immediate from \eqref{5}.
\end{proof}

\begin{theorem}\label{15}
For $(w,\s)$ as above, 
	$\Y$ is Cohen-Macaulay and normal. Moreover, $\Y$ is smooth (resp. Gorenstein) if and only if $\fX_w$ is. 
\end{theorem}
\begin{proof}
	To prove that $\Y$ is Cohen-Macaulay, it suffices to show that $\iota$
	in \eqref{3} is a regular immersion since $\fX_w$ is
	Cohen-Macaulay. 
	This is true since $\iota$ and $\iota_\s$ have the same
	codimension, 
	$\dim~\cB$, 
	by Proposition~\ref{15c}, and
	$\iota_\s$ is a regular immersion.
	
	For normality, it suffices to show that $\Y$ is regular in codimension
	one since it is Cohen-Macaulay. This follows from Proposition~\ref{5'}
	since every $v$ in $S_w$ satisfies
	$\ell(v)<\ell(w)-1$ by the normality of $\fX_w$. 
	
	Since $\Yvcell$ is nonempty for any $v$, 
	the smoothness criterion is also immediate from Proposition~\ref{5'}.
	Similar arguments apply for
	the Gorentein property if $\fX_w$ is Gorenstein at $\iota_\s(y)$, since as $\iota$ is a
	regular immersion
	(\cite[\href{https://stacks.math.columbia.edu/tag/0BJJ}{Tag
	0BJJ}]{stacks-project}). $G$-equivariant. \end{proof}

\subsection{Bott-Samelson resolutions of Lusztig varieties}
We recall the construction of the Bott-Samelson resolutions of Lusztig varieties and their properties. Using these resolutions, we complete the proof of Theorem~\ref{22}.

Let $\s \in G$.
Let $\underline w=(s_{i_1},\cdots, s_{i_\ell})$ be a word. 

\begin{definition}
Define $\bsY$ to be the fiber product
	\beq\label{32} \begin{tikzcd}
		\bsY  \arrow[r,hook,"\underline \iota"] \arrow[d]&\fX_{\underline{w}}\arrow[d] \\ 
		\cB\arrow[r,hook,"\iota_\s"]&\cB\times \cB
	\end{tikzcd} 
	\eeq
	where the right vertical map is \eqref{37}. Moreover, we define $\bsYcell:=\underline\iota^{-1}(\fX_{\underline w}^\circ)$.

For $1 \leq j \leq \ell$, put $E_j$ to be the preimage  of $\fD_j$ in $\bsY$, and denote
		\beq\label{34}\partial \bsY:=\bsY-\bsYcell=\bigcup_{j=1}^\ell E_j.\eeq  \end{definition}
Replacing $\iota_\s$ by $G\times \cB\to \cB\times \cB$ sending $(x,gB)$ to $(gB,xgB)$, we obtain 
\[Y_{\underline w}:=(G\times \cB)\times_{\cB\times \cB}\fX_{\underline w}\lra G\times \cB\xrightarrow{~pr_1~} G,\]
which was introduced and studied in \cite[2.6]{Lus85}. Its fiber over $\s$ is  $\bsY$. 

\smallskip
From now on, we always assume that $\s\in \Grs$ unless otherwise mentioned.
\begin{lemma}
	When $\s\in \Grs$, $\bsY$ is smooth of pure dimension $\ell$. 
\end{lemma}
\begin{proof}
	Since $\fX_{\underline w}\to \cB\times \cB$ is $G$-equivariant,
	Lemma~\ref{2'}, applied to \eqref{32}, implies that that the left exact sequence of tangent spaces
\[0\lra T_{\bsY,y}\lra T_{\fX_{\underline w},\underline \iota(x)}\oplus T_{\cB,\pi(y)}\lra T_{\cB\times \cB, \iota_\s(\pi(y))}\]
is exact from the right, for any $y\in \bsY$. So, $\dim~ T_{\bsY,y}=\ell$, the expected dimension of $\bsY$.  In particular, $\bsY$ is smooth of pure dimension $\ell$. \end{proof}

When $\underline w$ is a reduced word of $w$, by Lemma~\ref{40}(1), the vertical maps in \eqref{32} factor through those in \eqref{3}:
\beq\label{32a} \begin{tikzcd}
		\bsY  \arrow[r,hook,"\underline \iota"] \arrow[d,"\pi"']&\fX_{\underline{w}}\arrow[d]\\
		\Y\arrow[r,hook,"\iota"] \arrow[d,hook'] &\fX_w \arrow[d,hook'] \\ 
		\cB\arrow[r,hook,"\iota_\s"]&\cB\times \cB
	\end{tikzcd} 
	\eeq
	where the right vertical maps are isomorphisms over the open cells.
In this case, we call $\bsY$ the \emph{Bott-Samelson resolution} of $\Y$ associated to $\underline{w}$, as it resolves singularities of $\Y$, and $\pi$ is birational.

\begin{proposition}\label{25}
	Let $\underline w$ be a word and $\s\in \Grs$.
Then, the following hold.
	\begin{enumerate}
		\item When $\underline w$ is a reduced word of $w\in W$, the map $\pi$ in \eqref{32a} is an isomorphism $\bsYcell\xrightarrow{\cong} \Ycell$ over $\Ycell$.
		\item The boundary \eqref{34} is a simple normal crossing divisor with 
		\[E_j\cong \bsYjc \quad \text{ for all }1\leq j\leq \ell.\]
\item More precisely, for any $I\subset \{1,\cdots, \ell\}$, the image of $\bsYI\hookrightarrow \bsY$
	induced by \eqref{33} is the transverse intersection of $E_j$ with $j\notin I$.
	\end{enumerate}
\end{proposition}

\begin{proof}
	Immediate from Lemmas~\ref{40} and \ref{2'}. 
\end{proof}
As we will see in Proposition~\ref{26}, $\partial\bsY$ is an anti-canonical divisor of $\bsY$. In particular, $\bsY$ has a simple normal crossing anti-canonical divisor. This fact will play an important role in the next two sections.

\medskip

We end this section with the following, which is a direct consequence of Lemma~\ref{19}. 
This completes the proof of Theorem~\ref{22}.

\begin{theorem}\label{8}
	$\Y$ has rational singularities. 
\end{theorem}
\begin{proof}
	The proof is essentially the same as that for Schubert varieties given in \cite{Brion}. Let $\underline w$ be a reduced word of $w$.
	It suffices to show that $\pi_*\cO_{\bsY}\cong\cO_{\Y}$ and $R^i\pi_*\cO_{\Y}=0$ for $i>0$. The first isomorphism follows from the Zariski main theorem  \cite[III.11.4]{Har} since $\Y$ is normal by Theorem~\ref{15}. 
	
	For the vanishing, 
	consider the factorization of $\pi$ induced by Lemma~\ref{19}:
	\[f:~Y_{k+1}\hooklongrightarrow \hat Y_{k}\xrightarrow{~g~} Y_k\]
	for $1\leq k\leq \ell$, where $Y_k$ and $\hat Y_k$ are the fiber products of $\fX_k$ and $\hat \fX_k$ with $\iota$ over $\fX_w$ respectively. In particular, $Y_1=\Y$ and $Y_{\ell}=\bsY$.

	Let $\cI$ be the sheaf of ideals on $\hat Y_k$ defining $Y_{k+1}$. Applying $Rg_*(-)$ to
	$0\to \cI\to \cO_{\hat Y_k}\to \cO_{Y_{k+1}}\to 0$,
	we obtain an exact sequence
	\[R^ig_*\cO_{\hat Y_k}\lra R^if_*\cO_{Y_{k+1}}\lra R^{i+1}g_*\cI\]
	where $R^ig_*\cO_{\hat Y_k}=0=R^ig_*\cI$ for $i\geq 1$ since $g$ is a  $\PP^1$-bundle. Therefore, we obtain the vanishing $R^if_*\cO_{Y_{k+1}}=0$ for all $i>0$ and $k$. Since $\pi$ is the composition of all such $f$, finally $R^i\pi_*\cO_{\bsY}=0$ for $i>0$.
\end{proof}

Together with Proposition~\ref{15c} and Theorem~\ref{15},
this completes the proof of Theorem~\ref{22}.

\begin{remark}[Regular Lusztig varieties]\label{rem:regular}
	For regular $x\in G$, the Lusztig cell $Y_w^\circ(x)$ is nonempty and has pure dimension $\ell(w)$ for all $w\in W$, as shown by Springer (\cite[Proposition~5.1]{ellers-gordeev}) and in \cite[Lemma~2.1]{He-La}, respectively. 
	By the argument in the proof of Proposition~\ref{15c}, we then have $Y_w(x)=\overline{Y_w^\circ(x)}$. 
	
	Consequently, the Lusztig variety $Y_w(x)$ has pure dimension $\ell(w)$.
This also follows from a recent result of \cite{AN}, which shows that the corresponding Bott-Samelson varieties $Y_{\underline w}(x)$ are paved by affines 
	of dimension at most $\ell(\underline w)$, as every irreducible component of $Y_w(x)$ has dimension at least $\ell(w)$.
	
	In particular, the boundary
	$\partial Y_{\underline w}(x)$ is a Cartier divisor. 
\end{remark}

\bigskip

\section{Line bundles and cohomology} \label{sect ch 0}

In this section, we prove that the boundary $\partial \bsY$ is an anti-canonical divisor of $\bsY$ using the adjunction formula, and that it supports an ample divisor. In particular,  Lusztig cells are affine. Then, applying the Kawamata-Viehweg vanishing theorem, we prove Theorem~\ref{23} when $\mathrm{char}(k)=0$.
The proof   is motivated by that for $X_w$ given in \cite{Brion}.

\subsection{Canonical line bundles}
We compute the canonical line bundle $\omega_{\bsY}$ of $\bsY$, applying  the adjunction formula to the closed immersion $\underline\iota$ in \eqref{32}. 

Let $\rho \in X(T)$ be the half sum of all positive roots of $G$. The canonical line bundle of $\fX_{\underline{w}}$ is computed in \cite[Lemma~1.3]{kumar-canonical} or \cite[Proposition~2]{MR2} as
	\beq \label{47}\omega_{\fX_{\underline{w}}}=(L_{-\rho}\boxtimes L_{-\rho})|_{\fX_{\underline w}}\otimes \cO(-\partial \fX_{\underline{w}})\eeq
	where $A\boxtimes B$ means $pr_1^*A\otimes pr_2^*B$ for the $i$-th projections $pr_i$ of $\cB\times \cB$.

\begin{proposition}\label{26}
	$\omega_{\bsY}=\cO(-\partial \bsY)$.
	In particular,  the boundary divisor $\partial \bsY$ is a simple normal crossing anti-canonical divisor of $\bsY$.
\end{proposition} 

For  any regular closed immersion $\jmath$, let $N_\jmath$ denote  the normal bundle of $\jmath$. 

\begin{proof} 
	Applying the adjunction formula to \eqref{32}, we obtain
	\[\omega_{\bsY}=\underline\iota^*\omega_{\fX_{\underline{w}}}\otimes \det N_{\underline\iota}.\]
	 We now compute $\underline\iota^*\omega_{\fX_{\underline{w}}}$ and $\det N_{\underline\iota}$ separately.
	Applying $\underline\iota^*$ to \eqref{47}, we get
	\[\underline\iota^*\omega_{\fX_{\underline w}}= \pi^*\iota_\s^*(L_{-\rho}\boxtimes L_{-\rho})\otimes \underline\iota^*\cO(-\partial \fX_{\underline w})= \pi^*L_{-2\rho}(-\partial \bsY)\]
	since $\underline \iota^*\cO(\partial\fX_{\underline w})=\cO(\partial \bsY)$ by Proposition~\ref{25} and the pullback map $\iota_\s^*$ on the Picard group of $\cB\times \cB$ is the identity map.
	
	On the other hand, since $\underline\iota$ and $\iota_\s$ have the same codimension, the canonical subbundle map $N_{\underline \iota}\to \pi^* N_{\iota_\s}$ is an isomorphism. 
	Thus, we have
	\[\det N_{\underline\iota} = \pi^*\det N_{\iota_\s}= \pi^*L_{2\rho}\]
	where the second equality follows from $\det N_{\iota_\s}=\omega_{\cB}^{-1}=L_{2 \rho}$.

	Consequently, $\omega_{\bsY}=\cO(-\partial\bsY)$, so that  $\partial \bsY$ is an anti-canonical divisor of $\bsY$ and it is simple normal crossing by Proposition~\ref{25}.
\end{proof}

\smallskip

\subsection{Line bundles on Schubert varieties}
We provide 
a brief review of results from \cite{LT} and \cite{BS24} on line bundles on $X_{\underline w}$ and $\fX_{\underline w}$ respectively. 

Let $\underline w=(s_{i_1},\cdots,s_{i_\ell})$ be a word of $w$. Then, $X_{\underline w}$ is an iterated $\PP^1$-bundle
\beq\label{60}X_j':=P_{i_1}\times^B\cdots\times^BP_{i_j}/B~\stackrel{q_j}{\lra}~ X_{j-1}':=P_{i_1}\times^B\cdots\times^BP_{i_{j-1}}/B\eeq
starting from a point, as $j$ varies $1\leq j\leq \ell$. 
One can compose these maps with the usual multiplication map to produce line bundles:
for each $j$, let
\beq\label{62}\tau_j:~X_{\underline w}\xrightarrow{q_{j+1}\circ\cdots\circ q_\ell} X_j'\xrightarrow{~m_j~}  \cB\eeq
be the map sending $[p_1,\cdots,p_\ell]$ to $p_1\cdots p_j B$, and define the line bundle
\[M_j:=\tau_j^*L_{\varpi_{i_j}}\]
to be the pullback of the line bundle $L_{\varpi_{i_j}}$ associated to the $i_j$-th fundamental weight $\varpi_{i_j}$ of $T$.
Each $M_j$ has relative degree one (hence, it is relatively ample) with respect to the $j$-th $\PP^1$-bundle map in \eqref{60}.

\smallskip
Then, $M_1,\cdots, M_\ell$ freely generate $\Pic(X_{\underline w})$, and furthermore, they span the extremal rays of the ample cone of $X_{\underline w}$.

\begin{proposition}[\cite{LT}] \label{59}
Let $L$ be a line bundle on $X_{\underline w}$. The following holds.
\begin{enumerate}
	\item There exist unique $a_1,\cdots, a_\ell\in \Z$ such that $L=M_1^{\otimes a_1}\otimes \cdots \otimes M_\ell^{a_\ell}$;
	\item $L$ is ample if and only if $a_1,\cdots, a_\ell>0$. \end{enumerate}
\end{proposition}
\begin{proof}
	For the readers' convenience, we provide a sketch of the proof. 
	
	Note that the maps $q_j$ and $m_j$ in \eqref{60} and \eqref{62} fit into the fiber diagram
	\[\small\begin{tikzcd}
		X_j' :=P_{i_1}\times^B\cdots \times^B P_{i_j}/B  \arrow[r,"m_j"] \arrow[d,"q_j"']&G/B\arrow[d, "f_j" ']\\ X_{j-1}':=P_{i_1}\times^B\cdots \times^B P_{i_{j-1}}/B\arrow[r]&G/P_{i_j},
	\end{tikzcd}
	\begin{tikzcd}
	\left[p_1,\cdots, p_j\right] \arrow[r,|->] \arrow[d]&p_1\cdots p_jB\arrow[d,|->]\\ \left[p_1,\cdots,p_{j-1}\right]\arrow[r,|->]&p_1\cdots p_{j-1}P_{i_j}
	\end{tikzcd}\]
	where $B\subset P_{i_j}$ is the minimal parabolic subgroup corresponding to $s_{i_j}$ and $f_j$ is the natural projection. As $q_j$ and $f_j$ are $\PP^1$-bundles, every line bundle on $X_j'$ is written uniquely as the tensor product $q_j^*L'\otimes m_j^*L_{\varpi_{i_j}}^{\otimes a_j}$, for some $L'\in \Pic(X_{j-1}')$ and $a_j\in \Z$. 
	This proves (1), by the induction on $\ell$. 
	
	Moreover, since $m_j^*L_{\varpi_{i_j}}$ is relatively ample (of degree one) with respect to $q_j$ as well as globally generated, its complete linear system and $q_j$ define a closed immersion $X_j'\hookrightarrow  X_{j-1}'\times \PP H^0(X_j',m_j^*L_{\varpi_{i_j}})^*$ over $X_{j-1}'$. In particular, the line bundle $q_j^*L'\otimes m_j^*L_{\varpi_{i_j}}^{\otimes a_j}$ above is equal to the restriction of $L'\boxtimes \cO(a_j)$, so it is ample if $L'$ is and $a_j>0$ for all $j$. This proves the `if' part of (2) by the induction on $\ell$. 
	The `only if' part can be proved by restricting $L$ to the boundary divisors $D_i$. We omit the details, see \cite{LT}.  
\end{proof}

Analogous results hold for $\fX_{\underline w}$ with the same proof. 
Consider the map from $\fX_{\underline w}$ defined analogously to \eqref{60}:
\[\widetilde \tau_j:~\fX_{\underline w}\xrightarrow{\id_G\times \tau_j} G\times^B\cB\cong \cB\times \cB\xrightarrow{~pr_2~}\cB,\]
which sends $[g,p_1,\cdots, p_\ell]\mapsto gp_1\cdots p_jB$, and define the line bundle
\[\tM_j:=\widetilde \tau_j^*L_{\varpi_{i_j}}.\] 
By a slight abuse of notation, let $pr_1$ denote the first projection $\fX_{\underline w}\to\cB$.

\begin{proposition}[\cite{BS24}] \label{61}
	Let $L$ be a line bundle on $\fX_{\underline w}$. The following holds.
	\begin{enumerate}
		\item There exist unique $\lambda\in X(T)$ and $a_1,\cdots, a_\ell\in Z$ such that 
		\beq\label{64} L= pr_1^*L_\lambda \otimes \widetilde M_1^{\otimes a_1}\otimes \cdots \otimes \tM_\ell^{\otimes a_\ell}.\eeq
		\item $L$ is ample if and only if $\lambda$ is regular dominant and $a_1,\cdots, a_\ell>0$.
	\end{enumerate}
\end{proposition}

Propositions~\ref{59} and~\ref{61} particularly ensure that the ample cones of $X_{\underline w}$ and $\fX_{\underline w}$ remain unchanged when the base field $k$ varies, as the generators of the extremal rays are defined over $\Z$. This will be useful in the next section.

\medskip

There is another way to phrase Proposition~\ref{61}:
Since any line bundle $M\to X_{\underline w}$ is $B$-equivariant, it naturally extends to a line bundle 
\[\underline M:= G\times^BM\to G\times^BX_{\underline w}=\fX_{\underline w}\]
over $\fX_{\underline w}$.
This gives rise to a group homomorphism
\[\Phi:\Pic(\cB)\times \Pic(X_{\underline w})\lra \Pic(\fX_{\underline w}), \quad (L,M)\mapsto pr_1^*L\otimes \underline M.\]
One can easily check the following.
\begin{lemma} \label{lem:pullback}
	Let $1\leq j\leq \ell$. Then, 
	$\widetilde \tau_j^*L=\underline{\tau_j^* L}$ for any line bundle $L$ on $\cB$. In particular, $\tM_j=\underline{M_j}$.
\end{lemma}

\begin{proof}
	We prove the case $j=\ell$, as the proof for the other $j$ is the same.
	Note that   $L=L_\lambda$ for some $\lambda$. Then, $\tau_\ell^*L=P_{i_1}\times^B\cdots \times ^BP_{i_\ell}\times^B k_{-\lambda}$.
	Hence, $\underline {\tau_\ell^*L}=G\times^B(\tau_\ell^*L)$ fits into the fiber diagram
	\[\begin{tikzcd}
		\underline{\tau_\ell^*L}=G\times^BP_{i_1}\times^B\cdots \times^B P_{i_\ell}\times^B k_{-\lambda}  \arrow[r] \arrow[d]&G\times^B k_{-\lambda}=L\arrow[d]\\\fX_{\underline w}=G\times^BP_{i_1}\times^B\cdots \times^B P_{i_\ell}/B\arrow[r]&G/B=\cB
	\end{tikzcd}\]
	where the horizontal maps are those induced by  the multiplication map 
	$G\times \prod_{j=1}^\ell P_{i_j}\to G$, $(g,p_1,\cdots,p_\ell)\mapsto (gp_1\cdots p_\ell)$, so 
	 the lower map is $\widetilde \tau_\ell$.
\end{proof}

\medskip

Due to Lemma~\ref{lem:pullback},
Proposition~\ref{61} can be written simply as follows.
\begin{corollary}\label{65}
	The group homomorphism $\Phi$ is an isomorphism. Moreover, $\Phi(L,M)$ is ample if and only if $L$ and $M$ are ample.
\end{corollary}
\begin{proof}
	Immediate from Propositions~\ref{59} and \ref{61}, as any line bundle on $\fX_{\underline w}$ of the form \eqref{64} is equal to $\Phi(L_\lambda,M_1^{\otimes a_1}\otimes \cdots\otimes M_\ell^{\otimes a_\ell})$ by Lemma~\ref{lem:pullback}.
\end{proof}

\subsection{Ample divisors} 
Now we show that $\partial\bsY$ supports an ample divisor. To do this, we first
recall an analogous result for $X_{\underline w}$, on the Schubert side.

\begin{lemma} \label{lem:ample_Sch} Let $D_1,\cdots, D_\ell$ be the boundary
  divisors of $X_{\underline w}$.  Then: \begin{enumerate} \item There exist
    positive integers $a_1, \dots, a_{\ell}$ with  $\sum_{j=1}^{\ell}a_jD_j$
    ample. \item For each $1\leq k\leq \ell$, there exist positive integers
    $b_1, \dots, b_{\ell}$, depending on $k$, with $L_{b_k \rho} (\sum_{j\neq
    k} b_j D_j)$ ample. \end{enumerate} \end{lemma} \begin{proof} Note that we
    may allow $a_j$ to be rational number, since ampleness is invariant under
    scaling.
    
    (1) This is proved in \cite[Lemma~6.1]{LT}. One can check that each $D_j$
    is the pullback of a divisor with relative degree one (thus relatively
    ample) with respect to the $\PP^1$-bundle \eqref{60}. Hence, choosing
    $a_j>0$ to be sufficiently small compared to $a_i$ with $i<j$, the divisor
    $\sum_{j=1}^\ell a_jD_j$ is ample.
	
    (2)  The case $k=1$ is proved in \cite[Lemma~2.3.2]{Brion}. The proof for
    $k>1$ is almost the same, but we provide it here for the reader's
    convenience.
	
    Recall that the $\PP^1$-bundle map $q_r$ in \eqref{60} fits into the fiber
    diagram \[\begin{tikzcd} X_r' \arrow[r,"m_r"] \arrow[d,"q_r"']&G/B\arrow[d,
    "f_r" ]\\ X_{r-1}' \arrow[r]&G/P_{i_r}, \end{tikzcd} \] where $m_r$ is the
    multiplication map in \eqref{62} and $f_r$ is the natural projection.

    By abuse of notation, let  $D_j$, $1 \leq j \leq r$ be the boundary
    divisors on $X_r'$. 
    
    When $r<k$, applying the same argument in the proof of (1) to $q_r$, we can
    find an ample divisor $\sum_{j=1}^{k-1}b_jD_j$ on $X_{k-1}'$ with $b_j>0$
    for $1 \leq j \leq k-1$.
	
    When $r=k$, we claim that $m_r^*L_{b_k \rho}\otimes
    q_k^*\cO_{X_{r-1}'}(\sum_{j=1}^{k-1}b_jD_j)$ is ample on $X_{k}'$ for
    sufficiently small $b_k>0$. This follows from the decomposition
    $L_\rho=L_{\varpi_{i_k}}\otimes L_{\rho-\varpi_{i_k}}$. Indeed,
    $m_r^*L_{\varpi_{i_k}}$ is relative ample of degree one with respect to
    $q_r$ and $L_{\rho-\varpi_{i_k}}$ is the pullback of an ample line bundle
    on $G/P_{i_k}$. 
	
    When $r>k$, we can inductively construct an ample line bundle of the form
    we are interested in on $X_{r}'$, as follows. Suppose that $A:=
    m_{r-1}^*L_{b_k \rho}\otimes \cO_{X_{r-1}'}(\sum_{1\leq j\leq r-1, j\neq
    k}b_jD_j)$ is ample on $X_{r-1}'$. This pulls back to a globally generated
    line bundle on $X_r'$. Then, one can see that  it suffices to show that
    \beq\label{51}m_r^*L_{c\rho}\otimes (m_r^* L_{\rho}\otimes q_r^*m_{r-1}^*
    L_{-\rho}\otimes \cO_{X_r'}(D_r))^{\otimes b_k}\eeq is ample for
    sufficiently large $c>0$. Indeed, then by tensoring $q_r^*A$ with
    \eqref{51} we get an ample line bundle of the form
    $m_r^*L_{(c+1)\rho}\otimes \cO_{X_r'}(\sum_{1\leq j\leq r,j\neq k}b_jD_j)$
    where $b_r:=b_k$ and we can set $c+1$ to be our new $b_k$.
	
    On the other hand, $\omega_{X_r'}^{-1}=m_r^*L_\rho\otimes
    \cO_{X_r'}(\sum_{j=1}^r D_j)$ for every $r$. By the adjunction formula, the
    bundle in the parentheses in \eqref{51} is equal to
    $\omega_{q_r}^{-1}=m_r^*\omega_{f_r}^{-1}$. In particular, \eqref{51} is
    ample for any sufficiently large $c$.    \end{proof}

    Although $\bsY$ is not an iterated $\PP^1$-bundle like $X_{\underline w}$,
    we can deduce an analogous result for $\bsY$, combining Corollary~\ref{65}
    and Lemma~\ref{lem:ample_Sch}.

    \begin{proposition}\label{prop:ample_boundary} Let $E_1,\cdots, E_\ell$ be
      the boundary divisors of $\bsY$.. Then: \begin{enumerate} \item There
	exist positive integers $m_1, \dots, m_{\ell}$ with
	$\sum_{j=1}^{\ell}m_jE_j$ ample. \item For every $1\leq k\leq \ell$,
	there exist positive integers $n_1, \dots, n_{\ell}$,  depending on
	$k$, with $L_{n_k\rho}(\sum_{j\neq k}n_jE_j)$ ample. \end{enumerate}
	\end{proposition} \begin{proof} (1) By Lemma~\ref{lem:ample_Sch}(1),
	$D=\sum_{j=1}^{\ell}a_jD_j$ is ample for some $a_j>0$. Set
	\beq\label{eq:linebundle.M}M:=L_{-\rho}(mD)\eeq and suppose that $m$ is
	sufficiently large so that $M$ is also ample. Then, by
	Corollary~\ref{65}, $\Phi(L_\rho,M)=pr_1^*L_\rho\otimes \underline{M}$
	is ample on $\fX_{\underline w}$. Moreover, we have \[
	pr_1^*L_\rho\otimes \underline{M} =(L_\rho\boxtimes
      L_{-\rho})|_{\fX_{\underline w}}\otimes \cO_{\fX_{\underline w}}(m\fD),\]
      by applying Lemma~\ref{lem:pullback} with $j=\ell$ and $L=L_{-\rho}$,
      where $\fD:=\sum_{j=1}^\ell a_j\fD_j$. This restricts to
      $\cO_{\bsY}(\sum_{j=1}^\ell m_jE_j)$  on $\bsY$ where $m_j:=ma_j>0$.
      Since ampleness is preserved under restriction to subvarieties, it is
      ample.
		
    (2) The proof is similar: let $M := L_{b_k\rho}(\sum_{j\neq k} b_jD_j)$ as in
    Lemma~\ref{lem:ample_Sch}(2). Then the line bundle $\Phi(L_\rho,M)$ is ample by
Corollary~\ref{65}. Hence, its restriction to $\bsY$, which is
$L_{(b_k+1)\rho}(\sum_{j\neq k} b_jE_j)$, is also ample. \end{proof}
\begin{remark}\label{rem:indep} It follows from Propositions~\ref{59} and
  \ref{61} that the coefficients $a_j$  in Lemma~\ref{lem:ample_Sch} and $m$ in
  the proof of Proposition~\ref{prop:ample_boundary} are independent on the
  base field. In particular, the coefficients $m_j$ in
  Proposition~\ref{prop:ample_boundary} are also independent on the base field.
  This will be used in the next section. \end{remark}

  As a byproduct of Proposition~\ref{prop:ample_boundary}, we deduce the
  following.

  \begin{corollary}\label{cor:cell.affine} For any word $\underline w$ and $\s
    \in \Grs$, the open cell $\bsYcell$ is affine. In particular, the Lusztig
    cell $\Ycell$ is affine for any $w\in W$. \end{corollary} \begin{proof} By
    Proposition~\ref{prop:ample_boundary}(1), $\partial \bsY$ supports an ample
    divisor in $\bsY$. Since the complement of an ample divisor in a projective
    variety is affine, $\bsYcell=\bsY-\partial \bsY $ is affine. The second
    assertion follows from the first by taking $\underline w$ to be a reduced
    word of $w$, as in which case $\bsYcell\cong \Ycell$. \end{proof}

    \begin{remark}[Affineness of regular Lusztig
      cells]\label{rem:regular.affine} Corollary~\ref{cor:cell.affine} depends
      only on the facts that the preimage of $\partial \fX_{\underline w}$ in
      $\bsY$, which is $\partial \bsY$, is a Cartier divisor, and that
      \eqref{eq:linebundle.M} is ample on $ \fX_{\underline w}$. Therefore, the
      same proof shows that $Y_{\underline w}^\circ(x)$ and $Y_w^\circ(x)$ with
      $x$ \emph{regular} are also affine (cf.~Remark~\ref{rem:regular}).
    \end{remark}

    The same proof also establishes that the Deligne-Lusztig cells are affine.
    \begin{definition} Assume $\mathrm{char}(k)=p>0$. Let $q$ be a positive
      power of $p$,   and denote by $F$ the $q$-th Frobenius map. Given $w\in
      W$ and a word $\underline w$, the \emph{\DL variety} $\YDL$ and its
      \emph{Bott-Samelson resolution} $\bsYDL$ are defined as the following
      fiber products respectively. \[\begin{tikzcd}
	Y_w^{\mathrm{DL}}\arrow[r,hook] \arrow[d,hook'] &\fX_w \arrow[d,hook']
	\\ \cB\arrow[r,hook,"\iota_F"]&\cB\times \cB \end{tikzcd} \and
	\begin{tikzcd} Y_{\underline w}^{\mathrm{DL}}
	  \arrow[r,hook,"\underline \iota_F"]
	  \arrow[d]&\fX_{\underline{w}}\arrow[d]\\
	  \cB\arrow[r,hook,"\iota_F"]&\cB\times \cB \end{tikzcd} \] where
	  $\iota_F=(\id,F)$. Moreover, we define their open cells as
	  \[\YDLcell:=\iota_F^{-1}(\fX_w^\circ) \and (Y_{\underline
	  w}^\circ)^{\mathrm{DL}}:=\underline\iota_F^{-1}(\fX_{\underline
	w}^\circ).\]	\end{definition} Note that, when $\underline w$ is a
	reduced word, the map $Y_{\underline w}^{\mathrm {DL}}\to \cB$ has the
	image $\YDL$, and its restricts to an isomorphism $(Y_{\underline
	w}^\circ)^{\mathrm{DL}}\xrightarrow{\cong} \YDLcell$.

	\begin{corollary}[Affineness of Delinge-Lusztig cells]\label{cor:DL}
	  The open cell $(Y_{\underline w}^\circ)^{\mathrm{DL}}$ is affine for
	  any word $\underline w$. In particular, $\YDLcell$ is affine for any
	  $w\in W$. \end{corollary} \begin{proof} It is proved in
	  \cite[Lemma~9.11]{DL} that the image of $\iota_F$ intersects any
	  $G$-orbit in $\cB\times \cB$ transversely. Hence $\bsYDL$ is smooth,
	  and its boundary \[\partial \bsYDL:=\bsYDL-(Y_{\underline
	  w}^\circ)^{\mathrm{DL}}=\underline\iota_F^{-1}(\partial\fX_{\underline
	w})\] is a simple normal crossing divisor $\bigcup_{j=1}^\ell
	E_j^{\mathrm{DL}}$ with
	$E_j^{\mathrm{DL}}:=\underline\iota_F^{-1}(\fD_j)\cong Y_{\underline
	w_{[j]^c}}^{\mathrm{DL}}$. Therefore, $ (Y_{\underline
      w}^\circ)^{\mathrm{DL}}$ is affine if $\sum_j m_j E_j^{\mathrm{DL}}$ is
      ample for some $m_j>0$. 
	
      On the other hand, for $M$ defined in \eqref{eq:linebundle.M} with
      sufficiently large $m$, \[ \cO_{\bsYDL}\Big(\sum_{j=1}^\ell
      m_jE_j^{\mathrm{DL}}\Big)\cong \underline\iota_F^*\Phi(L_{q\rho},M), \]
      and this is ample. Indeed, $L_{q\rho}$ and $M$ are ample, thus
      $\Phi(L_{q\rho},M)$ and its restriction by $\underline\iota_F$ are ample.
      Therefore, $ (Y_{\underline w}^\circ)^{\mathrm{DL}}$ is affine. The
      second assertion follows from the first by taking $\underline w$ to be a
      reduced word of $w$. \end{proof}

To the best of our knowledge, this problem has been open since the introduction
of $\YDL$ in \cite{DL}. 
We refer the reader to
\cite{DL,bonnafe-rouquier-affine,he-affine,orlik-rapoport, he-lusztig, harashita-affine} for
previous results on the affineness of $(Y_{w}^\circ)^{\mathrm{DL}}$ under
assumptions on $w$ or $q$.
However, we give a brief explanation the original motivation for the question 
along with some of the history.

\subsection*{Brief history of affineness for Deligne-Lusztig varieties}
\label{s-history}
Suppose $G$ is a reductive group over a finite field $\mathbb{F}_q$. 
By~\cite[Corollary 1.14]{DL}, the set of $G(\mathbb{F}_q)$-conjugacy 
classes of maximal tori is in one-one correspondence with the 
set $W^{\natural}_F$ of Frobenius twisted conjugacy classes in the Weyl group.
Using this correspondence, write $T_w$ for a maximal torus corresponding
to the Weyl group element $w$. 
In~\cite[Definition 1.9]{DL}, Deligne and Lusztig defined a virtual representation of 
$R^{\theta}(w)$ of the finite group $G(\mathbb{F}_q)$  
attached to a Weyl group element $w$ of $G$ and a character $\theta$ of 
$T_w(\mathbb{F}_q)$.
By construction, $R^{\theta}(w)$ is the alternating sum 
\begin{equation}
  \label{e-dlcoh}
\sum_i (-1)^i H^i_c ((Y_{w}^\circ)^{\mathrm{DL}}, \mathcal{F}_{\theta})
\end{equation}
of representations 
of $G(\mathbb{F}_q)$ on the compactly supported cohomology with coefficients 
in a $\overline{\mathbb{Q}}_{\ell}$-local system $\mathcal{F}_{\theta}$ corresponding to the character $\theta$.

Deligne and Lusztig showed that $\pm R^{\theta}(w)$ is irreducible when $\theta$
is in general position.  
Moreover, 
assuming that  
$(Y_{w}^\circ)^{\mathrm{DL}}$ is affine, 
they showed that, for $\theta$ in general position, the cohomology groups 
in~\eqref{e-dlcoh} vanish unless $i = \ell(w)$~\cite[Corollary 9.9]{DL}.
In other words, they showed that 
\begin{equation}
  \label{dl-van}
  \theta\text{ in general position \& } i \neq \ell(w) \Rightarrow 
H^i_c ((Y_{w}^\circ)^{\mathrm{DL}}, \mathcal{F}_{\theta}) = 0.
\end{equation}
Thus, assuming the affineness of $(Y_{w}^\circ)^{\mathrm{DL}}$, 
they showed that 
$
H^{\ell(w)}_c((Y_{w}^\circ)^{\mathrm{DL}}, \mathcal{F}_{\theta})
$
is an explicit model of the irreducible representation $(-1)^{\ell(w)}R^{\theta}(w)$.

Deligne and Lusztig showed that
$(Y_{w}^\circ)^{\mathrm{DL}}$ is affine whenever a 
certain combinatorial criterion for $w$ in terms of the root system of $G$ 
holds~\cite[Theorem 9.7]{DL}.
In particular, they deduced that $(Y_w^\circ)^{\mathrm{DL}}$ is affine whenever $q$ is larger than the Coxeter number of $G$.
They also explained that, if $w$ and $w'$ are two Frobenius conjugate elements
of the Weyl group, then~\eqref{dl-van} holds for $w$ if and only if~\eqref{dl-van}
holds for $w'$~\cite[Remark 9.15.1]{DL}.
Using these ingredients, they showed that~\eqref{dl-van} holds for all classical 
groups and for $G_2$.

The papers~\cite{bonnafe-rouquier-affine,he-affine,orlik-rapoport} prove that 
$(Y_{w}^\circ)^{\mathrm{DL}}$ is affine if $w$ is a minimal length element 
in its Frobenius conjugacy class.
By the remark of Deligne and Lusztig mentioned above, this suffices to show 
that~\eqref{dl-van} always holds for $\theta$ in general position. 

\subsection{Cohomology vanishing}\label{ss:vanishing}
We prove Theorem~\ref{23} using the following  version of Kawamata-Viehweg vanishing theorem, stated in \cite[Corollary~5.12.d]{EV}. In this subsection, we assume that $k$ has characteristic zero. 
\begin{theorem}\label{thm:kawamata-viehweg}
	Let $X$ be a smooth projective variety (over a field of characteristic zero) and let $M$ be a line bundle over $X$. Let $N>0$ be a positive integer and  let $B=\sum_{j=1}^\ell c_jB_j$ be a normal crossing divisor with $0<c_j<N$. If $M^{\otimes N}(-B)$ is big and nef (e.g. ample), then  
	$H^i(X,M\otimes\omega_X)=0$ for $i>0$.
\end{theorem}

The proof of the following is motivated by that for $X_w$ given in \cite{Brion}.

\begin{theorem}\label{thm:vanishing}
	Let $L$ be a line bundle on $\Y$.
	\begin{enumerate}
		\item If $L$ is nef, then $H^i(\Y,L)=0$ for $i>0$.
		\item If $L$ is ample and $v\in W$ satisfies $v<w$, 
		then the restriction map $H^0(\Y,L)\to H^0(\Yv,L)$ is surjective.
	\end{enumerate}
\end{theorem}
\begin{proof}
	(1)
	Let $\underline w$ be a reduced word of $w$ with $\ell=\ell(w)$.
	By Proposition~\ref{prop:ample_boundary}(1), $A:=\sum_{j=1}^{\ell} m_jE_j$ is ample for some $m_j>0$.
	Let $B:=N\partial \bsY-A$ 
	with $N\geq m_j$ so that $B$ is effective.
	Then, for $M:=\pi^*L\otimes \omega_{\bsY}^{-1}$,
	\[
	M^{\otimes N}(-B)
	=\pi^*L^{\otimes N}\otimes \cO(A)\]
	is ample. 
	Hence, by Theorem~\ref{thm:kawamata-viehweg} and the
	projection formula, 
	\[H^i(\Y,L)=H^i(\bsY,\pi^*L)=H^i(\bsY,M\otimes \omega_{\bsY})=0\] 
	for $i>0$, as desired. 
	
	(2) 
	We may assume that $\ell(v)=\ell(w)-1$. Then, there exists a word $\underline w$ of $w$ and $k$ such that $\underline v:=\underline w_{\{k\}^c}$ is a reduced word of $v$ (\cite[Corollary~2.2.2]{Brion}).  In particular, $\bsYv\cong E_k$.
	Hence, it suffices to show that $H^1(\bsY, \pi^*L(-E_k))=0$, since the assertion will follow from this by the projection formula. 
	
	By Proposition~\ref{prop:ample_boundary}(2), 
	$A:=\pi^*L_{n_k\rho}(\sum_{j\neq k}^\ell n_jE_j)$ is ample on $\bsY$ for some $n_j>0$.
	Let $B:=\sum_{j\neq k}(N-n_j)E_j$ with $N$ sufficiently large. Then, for the line bundle
	$M$ defined by $M\otimes \omega_{\bsY}=\pi^*L(-E_k)$, we have
	\[M^{\otimes N}(-B)=\pi^*(L^{\otimes N}\otimes L_{-n_k\rho})\otimes A,\]
	and this
	is ample for sufficiently large $N$.
	Hence, it follows from Theorem~\ref{thm:kawamata-viehweg} that
	\[H^1(\bsY, \pi^*L(-E_k))=H^1(\bsY,M\otimes \omega_{\bsY})=0.\]
\end{proof}

\begin{corollary}\label{cor:vanishing}
	If $\lambda$ is dominant, then $H^i(\Y,L_\lambda)=0$ for $i>0$. 
	If $\lambda$ is regular dominant and $v<w$, then $H^0(\Y,L_\lambda)\to H^0(\Yv,L_\lambda)$ is surjective.
\end{corollary}
\begin{proof}
	This is immediate from Theorem~\ref{thm:vanishing}, since $L_\lambda$ is nef (resp. ample) on $\Y$ if $\lambda$ is (resp. regular) dominant.
\end{proof}
\begin{remark}\label{rem:vanishing}
	(1) Theorem~\ref{thm:vanishing} is more general than Corollary~\ref{cor:vanishing}, as the restriction $\Pic(\cB) \to \Pic(\Y)$ is not surjective in general. For example, when $G=\GL_3$ and $w=231\in S_3$, $\Y$ is isomorphic to the permutohedral toric surface, as it is the closure of an orbit of a 2-torus $C_G(\s)\cong \mathbb{G}_m^2$.
    In particular, $\Pic(\cB)$ has rank two while $\Pic(\Y)$ has rank four.
    
    (2) The surjectivity is not necessarily true for nef line bundles. For example, unless $w= e$, $H^0(\Y,\cO_{\Y})\to H^0(\Ye,\cO_{\Ye})=k^{\lvert W \rvert}$ is not surjective. 
\end{remark}

Using Frobenius splittings,
in the next section, 
we will extend Theorem~\ref{thm:vanishing}  
to positive characteristic.

\bigskip

\section{Frobenius splitting}
\label{s.cohomology}
We assume throughout this section that $k$ has positive characteristic. We show that regular semisimple Lusztig varieties are Frobenius split, using Proposition~\ref{26}. This implies Theorem~\ref{23} in positive characteristic.

By semicontinuity of cohomology, this also yields an alternative proof of Theorem~\ref{23} in characteristic zero.

\subsection{Preliminaries on Frobenius splittings}
In this subsection, we review the Frobenius splitting method, as introduced by
Mehta-Ramanathan (\cite{MR}) and  refined by Ramanan-Ramanathan (\cite{RR85},
\cite{R87}). For further studies, see also \cite{MR}, \cite{R87}, \cite{BI94},
and \cite{BK}.

Let $k$ be an algebraically closed field of characteristic $p>0$. 
Let $X$ be a variety over $k$ and let $F: X \rightarrow X$ be the absolute Frobenius morphism. Then we have a natural map $\mathcal O_X \rightarrow F_*\mathcal O_X$, the $p$-th power map, and thus
an exact sequence
\begin{equation} \label{e.frobenius exact}
0 \rightarrow \mathcal O_X \rightarrow F_* \mathcal O_X \rightarrow C \rightarrow 0.
\end{equation}
 We say that $X$ is {\it Frobenius split} if the exact sequence (\ref{e.frobenius exact}) 
splits. 
Moreover, a closed subvariety $Y$ of $X$ is said to be \emph{compatibly split} in $X$ if there exists a splitting $\varphi: F_*\cO_X\to \cO_X$ be a Frobenius splitting of $X$ such that $\varphi(F_*I)=I$ for the ideal sheaf $I$ defining $Y$ in $X$, so that it induces a splitting of $Y$. 

\smallskip
Not every projective variety is Frobenius split. But once it is, it has many interesting geometric properties, one of which is the vanishing of the cohomology of line bundles. 

\begin{proposition} [Propositions 1--3 of \cite{MR}] \label{prop:Frob.basic}
Let $X$ be a projective variety and $L$ be an ample line bundle on $X$. Assume that $X$ is Frobenius split. 
\begin{enumerate} 
\item $H^i(X,L)=0$ for $i>0$. 
\item If $X$ is smooth and pure-dimensional, then $H^i(X, L^{-1})=0$ for $i <\dim \,X$. 
\item If $Y$ is a closed subvariety of $X$ which is compatibly split in $X$, then the restriction $H^0(X,L)\to H^0(Y,L)$ is surjective.
\end{enumerate}
\end{proposition} 
For example, the splitting induces an injection $H^i(X,L)\hookrightarrow H^i(X,L^{\otimes p})$ for any line bundles $L$, so (1) follows from the Serre vanishing theorem.

Examples of projective varieties that are Frobenius split include flag varieties, Schubert varieties, spherical varieties, and the cotangent bundles of flag varieties. 

A problem of finding a Frobenius splitting of a normal variety can be reduced to that of its resolution of singularities, due to the following lemma.

\begin{lemma}[Proposition~4 of \cite{MR}]\label{lem:push.Frob}
	Let $f:X \to S$ be a proper morphism of algebraic varieties with $f_*\cO_X=\cO_S$.
	Then, we have:
	\begin{enumerate}
		\item If $X$ is Frobenius split, then $S$ is also Frobenius split.
		\item If $Y$ is a closed subvariety of $X$ which is compatibly split in $X$, then its image $f(Y)$ is compatibly split in $S$.
	\end{enumerate}
\end{lemma}
When $X$ is smooth, due to Serre duality and projection formula, we get an isomorphism between $H^0(X, \omega_X^{1-p})$ and $\Hom(F_* \cO_X,\cO_X)$. 
Thus we can express criteria for Frobenius splitting by using the existence of certain sections of the anti-canonical line bundle.

\begin{proposition}[Propositions~7 and 8 of \cite{MR}]  \label{p.criteria splitting} Let $X$ be a smooth projective variety of dimension $n$ over $k$.
Assume that there is a section $s$ of $\omega_X^{-1}$ and $x \in X$ such that ${\rm div} s$ is $Z_1 + \dots + Z_n+E$ with $x \in \bigcap_{j=1}^n Z_j$ and $x \not \in E$, where $ Z_1+ \dots + Z_n $ is simple normal crossing and $E$ is effective. Then $X$ is Frobenius split, with a splitting given by $s^{p-1}\in H^0(X,\omega_X^{1-p})$.

 Moreover, if further $Z_j$ are irreducible, then the Frobenius splitting of $X$ given by $s^{p-1}\in H^0(X,\omega_X^{1-p})$ induces a compatible splitting of the intersection $\bigcap_{j\in I}Z_j$ for any nonempty subset $I\subset \{1,\cdots,n\}$. 
\end{proposition} 

Let $D$ be an effective divisor of $X$. We say that $X$ is {\it Frobenius $D$-split} if there is a Frobenius splitting $F_*\mathcal O_X \rightarrow \mathcal O_X$ which factors through $F_*(\mathcal O_X(D)) \rightarrow \mathcal O_X$.

\begin{proposition} [Section 2 of \cite{RR85}, Proposition 1.13 of \cite{R87}] \label{prop frob ample}  Let $X$ be a projective variety which is Frobenius $D$-split with an ample divisor $D$. If $L$ is a nef line bundle on $X$, then $H^i(X,L)$ vanishes for $i >0$.
\end{proposition}

\begin{proposition} [Section 1.3  and Section 1.16 of \cite{R87}]  \label{prop properties} $\,$ 

\begin{enumerate}
\item
If $X$ is Frobenius $D$-split and $D'$ is another (Cartier) divisor with $0 \leq D' \leq D$, then $X$ is Frobenius $D'$-split.

\item If $X$ is smooth and Frobenius split, then it is $\omega_X^{1-p}$-split: Any section of $\omega_X^{1-p}$ giving a splitting vanishes on a divisor whose corresponding line bundle is $\omega_X^{1-p}$.
\end{enumerate}

\end{proposition}

\subsection{Applications to Lusztig varieties}
In this subsection, we prove the existence of Frobenius splittings on regular semisimple Lusztig varieties.

Throughout this subsection, $\underline w$ is a reduced word of $w\in W$ and $\ell=\ell(w)$.

\begin{theorem} \label{28}
	Let $k$ be of characteristic $p>0$. Then, $\bsY$ is Frobenius split, compatibly with $\bsYI$ for any subset $I\subset\{1,\cdots, \ell\}$. Hence, $\Y$ is Frobenius split, compatibly with $\Yv$ for any $v<w$.   
	In particular, for any ample line bundle $L$ on $\Y$,  the following hold: 
	\begin{enumerate}
		\item $H^i(\Y,L)=0$ for $i>0$,
		\item the restriction map $H^0(\Y,L)\to H^0(\Yv,L)$ is surjective.
	\end{enumerate}
\end{theorem}
\begin{proof} By Proposition \ref{26}, 
	the canonical line bundle of $\bsY$ is
	\[\omega_{\bsY}=\cO(-\partial \bsY)\]
	where $\partial \bsY=\bigcup_{j=1}^{\ell} E_j$ is a simple normal crossing divisor.
	By Proposition~\ref{25}(3), $\bigcap_{j=1}^{\ell}E_j$ is the fiber product of $\cB\xrightarrow{\iota_\s}\cB\times \cB\xleftarrow{}\bigcap_{j=1}^{\ell}\fD_j\cong \cB$, where the right map is the diagonal map. In particular, $\bigcap_{j=1}^{\ell}E_j\cong \Ye\cong W$, so it is nonempty. 
	Hence, by Proposition~\ref{p.criteria splitting}, $\bsY$ is Frobenius split, where a splitting $\varphi$ is given by $s^{p-1}$ and $s$ is the anti-canonical section with $\mathrm{div} s=\partial \bsY$.
	 
	On the other hand, as every irreducible component of $E_j$ passes through
	a point in $\bigcap_{j=1}^{\ell}E_j$, it is compatibly split with
	$\varphi$  by Proposition~\ref{p.criteria splitting}. Consequently, any
	subvariety in $\bsY$ obtained by taking the union or intersection of
	irreducible components of $E_j$ for varying $j$ is also compatibly
	split with $\varphi$ (\cite[Proposition~1.9]{R87}). In particular,
	$\bsYI$ is compatibly split in $\bsY$ for any subset
	$I\subset\{1,\cdots, \ell\}$. Now it follows from
	Lemma~\ref{lem:push.Frob} and the fact that every $v<w$ has a word of
	the form $\underline w_I$ that $\Y$ is Frobenius split, compatibly with
	$\Yv$ for any $v<w$. 
	
	The assertions (1) and (2) are immediate from Proposition~\ref{prop:Frob.basic}.
\end{proof}

By Proposition~\ref{prop properties}(2) and Theorem~\ref{28}, we have the following.
\begin{proposition}\label{58}
	$\bsY$ is Frobenius $(p-1)\partial \bsY$-split.
\end{proposition}

Using this, we can improve Theorem~\ref{28}(1) to the case of nef line bundles under the assumption that $p$ is sufficiently large for given $(\underline w,\s)$, as follows.
\begin{theorem} \label{thm:vanishing.Frob}
	Let $k$ be of characteristic $p>0$. When $p$ is sufficiently large, $\bsY$ is Frobenius  $D$-split  for some ample divisor $D$. 
 In particular, for a nef line bundle $L$ on $\Y$,  $H^i(\Y,L)=0$ 
 for $i>0$. 
\end{theorem}
\begin{proof}
	By Proposition~\ref{prop:ample_boundary}, there exists 
	an ample divisor $D:=\sum_{j=1}^\ell m_jE_j$ with positive integers $m_j$ which are independent of base field (Remark~\ref{rem:indep}).
	If $p$ is sufficiently large so that $p>m_j$ for all $j$, then $(p-1)\partial \bsY \geq D$. Hence, $\bsY$ is Frobenius $D$-split by Propositions~\ref{prop properties}(1) and~\ref{58}. The second assertion then immediately follows from Proposition~\ref{prop frob ample}.
\end{proof}

\begin{corollary}\label{cor:vanishing.Frob}
	Let $k$ be of characteristic $p>0$. For $\lambda$ regular dominant, 
	\begin{enumerate}
		\item $H^i(\Y, L_\lambda)=0$ for $i>0$,
		\item If $v<w$, then $H^0(\Y,L_\lambda)\to H^0(\Yv, L_\lambda)$ is surjective.
	\end{enumerate}
	 When $p$ is sufficiently large (so that $p>m_j)$, \emph{(1)} holds also for $\lambda$ dominant.
\end{corollary}
\begin{proof}
	Immediate from Theorem~\ref{thm:vanishing.Frob}.
\end{proof}

\subsection{Cohomology vanishing in characteristic 0}\label{20}
An alternative proof of Corollary~\ref{cor:vanishing} can be given using Corollary~\ref{cor:vanishing.Frob} and the semicontinuity of cohomology. Accordingly, we restrict here to the case $L=L_\lambda$ with $\lambda$ dominant.

Consider the Chevalley group $G_\ZZ$ over $\ZZ$, which specializes to $G$ over $k$ (see \cite{Jan}). Suppose that $k$ is an algebraically closed field of characteristic zero, and $\s\in G(k)$ is regular semisimple. Since $G_\Z$ is of finite type over $\Z$, the $k$-point $\s:\Spec(k)\to G_\Z$ extends to 
$\Spec(R)\to G_\Z$
over $\Z$, for some finitely generated $\Z$-subalgebra $R\subset k$ with $K(R)=k$.  
Since regular semisimpleness is an open condition (cf.~\cite[2.14]{Ste}), there exists a nonempty open subscheme $U\subset \Spec(R)$ such that the induced section
\[\s_U:U\lra G_U\]
with $G_U:=G_\Z\times_\Z U$ sends every geometric point to a regular semisimple element.
Similarly, we denote $\cB_U:=\cB_\Z\times_\Z U$ and $(\fX_w)_U:=(\fX_w)_\Z\times_\Z U$ for $w\in W$, where $\cB_\Z:=G_\Z/B_\Z$ and $(\fX_w)_\Z \subset \cB_\Z$ denote the flag variety and relative Schubert varieties over $\Z$. 

Define the Lusztig variety associated to $(w,\s_U)$ 
over $U$  
as the fiber product
\[\begin{tikzcd}
	(\Y)_U \arrow[r,hook,"\iota_U"]\arrow[d,hook']& (\fX_w)_U \arrow[d,hook']\\ \cB_U \arrow[r,hook,"\iota_{\s_U}"]&\cB_U\times_U\cB_U
\end{tikzcd}\]
analogous to \eqref{3},
which specializes to $\Y$ over $k$.

Since each geometric fiber of  $(\Y)_U$ over $U$ has expected codimension $\mathrm{codim}~\iota_{\s_U}=\dim~ \cB$ in the corresponding fiber of $(\fX_{w})_U$, and since $(\fX_{w})_U$ is flat over $U$, it follows that $(\Y)_U$ is also flat over $U$ (\cite[\href{https://stacks.math.columbia.edu/tag/0470}{Tag 0470}]{stacks-project}). 

Since $U$ is a scheme locally of finite type over $\Z$ with nonempty generic fiber $k$, there exists a geometric point $u\in U$ lying over any given sufficiently large prime $p$. Thus, Theorem~\ref{thm:vanishing.Frob} applies to $(\bsY)_u:=(\bsY)_U\times_U u$. Hence $H^i((\Y)_u,L_\lambda)=0$ for all $i$ and $\lambda$ dominant. 

Applying the semicontinuity theorem (\cite[Theorem~12.8]{Har}) to $(\Y)_U\to U$ and the restriction of the  line bundle $(L_\lambda)_U=G_U\times_{B_U}k_{-\lambda}$ over $\cB_U$ to $(\Y)_U$, we conclude that $H^i(\Y,L_\lambda)=0$ over $k$. 
Similarly, we can deduce the surjectivity of $H^0(\Y,L_\lambda)\to H^0(\Yv,L_\lambda)$ for $\lambda$ regular dominant and $v<w$.
\hfill $\square$
\begin{remark}
	We restrict to the case $L=L_\lambda$ here, as we do not know 
    whether the nefness of line bundles is an open condition in projective families (cf.~\cite[Proposition~1.4.14]{Laz}).
\end{remark}

\bigskip
\section{Degeneration to Hessenberg varieties}\label{sec:degen}

In this section, we investigate relationships of regular semisimple Lusztig varieties to regular semisimple Hessenberg varieties. 
More precisely, we show that Lusztig varieties degenerate to Hessenberg varieties, suggesting that Hessenberg varieties can be viewed as linear versions of Lusztig varieties.

\smallskip

Let $G$ be a simple linear algebraic group and $(B,T)$ be a Borel subgroup and a maximal torus. 
Let $\Phi^+$ be the set of positive roots and $\Delta=\{\alpha_1, \dots, \alpha_{r}\}$ be the set of simple roots.  
Let $W$ be the Weyl group and $S=\{s_1, \dots, s_r\}$ be the set of simple reflections corresponding to $\Delta$. 
Let $\mathfrak g, \mathfrak b$ and $\mathfrak t$ be the Lie algebra of $G$, $B$ and $T$ respectively. 

Let $\grs\subset \fg$ be the open subset of regular semisimple elements, and denote its image in the projectivization $\PP(\fg)$ by $\pgrs$. 

In this section, we assume that $k$ is the field of complex numbers $\C$, as we will consider the singular cohomology of Lusztig varieties.
\subsection{Hessenberg varieties} \label{ss:Hess space}
A {\it Hessenberg space} is a $B$-invariant subspace of $\mathfrak g$ containing $\mathfrak b$. 
For $s\in \fg$ and a Hessenberg space $H\subset \fg$, 
the {\it Hessenberg variety}
associated to $(s,H)$
is defined to be a subvariety of $\cB$ 
\begin{equation} \label{e.Hess}
X_H(s):=\{gB \in \cB: g^{-1}.s \in H\}
\end{equation}
where $.$ denotes the adjoint action of $G$ on $\fg$.
Note that \eqref{e.Hess} is isomorphic to the fiber of $s$ under the map $G\times^BH\to \fg$ sending $[g,y]\mapsto g.y$.
\smallskip

We introduce a slight generalization of these definitions, where $H$ is not necessarily a linear subspace of $\fg$. 
When $H$ is a linear subspace of $\fg$, we call it a \emph{linear} Hessenberg space to refer to the original definition.

\begin{definition} \label{d.Hess} Consider the natural projection $\pi_\fb:\fg\to \fg/\fb$. Let $s\in \fg$.
	\begin{enumerate}
		\item We define a \emph{Hessenberg space} to be a nonempty $B$-invariant closed subscheme $H\subset \fg$ of pure dimension, which is saturated with respect to $\pi_\fb$, (that is, $H=\pi_\fb^{-1}(H/\fb)$ for some $B$-invariant closed subscheme in $\fg/\fb$ of pure dimension, which we denote by  $H/\fb$).
		\item We define the \emph{Hessenberg scheme $X_H(s)$ associated to $(s,H)$} to be the (scheme-theoretic) fiber of $s$ under $G\times^BH\to \fg, ~[g,h]\mapsto g.h$. 
			
			Note that this is equivalent to define $X_H(s)$ to be the fiber product
		\beq\label{75}
		\begin{tikzcd}
			X_H(s)\arrow[r,hook]\arrow[d,hook']&G\times^B(H/\fb)\arrow[d,hook']\\ \cB\arrow[r, hook,"\iota_s"]&G\times^B(\fg/\fb)=T_\cB
		\end{tikzcd} 
		\eeq
		where $\iota_s$ denotes the section of $T_\cB\to\cB$ given as  $gB\mapsto [g,g^{-1}.s]$.
	\end{enumerate}
\end{definition}
\smallskip
In this generalization, the scheme $X_H(s)$ is not necessarily smooth, even when $s$ is regular semisimple. Rather, its singular locus is determined by that of $H/\fb$. 

\begin{proposition}\label{77} Let $s\in \grs$. Then, the following hold.
\begin{enumerate}
	\item The image of $\iota_s$ intersects transversely with any $G$-orbits in $T_\cB$.
	\item If $H/\fb$ is a $B$-invariant closed subscheme in $\fg/\fb$, 
		then $X_H(s)$ also has pure dimension $\dim~H/\fb$.
	\item The singular locus of $X_H(s)$ is the preimage of $G\times^B\mathrm{Sing}(H/\fb)$. 
	\item If a map $V\to H/\fb$ is a $B$-equivariant resolution of singularities, then so is the induced map $X_H(s)\times_{G\times^B(H/\fb)}G\times^BV\to X_H(s)$. 
\end{enumerate}

\end{proposition}
\begin{proof}
	(1) This is an analogue to Lemma~\ref{2'} and it is essentially the content of \cite[Theorem~6]{demari-procesi-shayman} that establishes the smoothness of $X_H(s)$ in linear case. 
	
	The remaining assertions follow by the same arguments as in the proofs of Propositions~\ref{15c} and \ref{5'} and of the smoothness of $\bsY$ respectively. 
	
	(2) We may assume that $H$ is irreducible and reduced. Then, there is a $G$-equivariant stratification $Z:=G\times^B (H/\fb)=\bigsqcup_{j\geq0}Z_j$ of $Z$ with $Z_j$ being (not necessarily connected) locally closed smooth $G$-invariant subvarieties of pure codimension $j$ with $\mathrm{Sing}(Z)=Z-Z_0$. Indeed, one may take $Z_0$ to be the smooth locus of $Z$, take the next largest nonempty $Z_j$ to be the smooth locus of $Z-Z_0$, and so on. Now let $\Delta_s$ be the image of $\iota_s$. Then, $X_H(s)$ is 
	\[\Delta_s\cap Z=\bigsqcup_{j\geq0}\Delta_s\cap Z_j.\]
	By the LHS, every component of $X_H(s)$ has dimension at least $\ell$ while by the RHS and (1), it is the union of locally closed (smooth) subvarieties of dimension $\ell-j\leq \ell$. Since $X_H(s)$ is nonempty, it has pure dimension $\ell$.

	(3) The short exact sequences of tangent spaces, analogous to those used in the proof of Proposition~\ref{5'}, show that for any $x\in X_H(s)$, we have $\dim~T_{Z,\iota_s(x)}-\dim~T_{X_H(s),x}=\dim~\fg/\fb$ which equals $\dim~Z-\dim~X$ by (2). 
	In particular, $x\in X_H(s)$ is smooth if and only if $\iota_s(x)\in Z$ is smooth.
	
The proof for (4) is similar as well. 
\end{proof}

Especially when $H$ is a cone, the natural scaling $\mathbb{G}_m$-actions on $H$ and $\fg$ make the map $G\times^BH\to \mathfrak g$ a $\mathbb{G}_m$-equivariant morphism.
In particular, $X_H(s)=X_H(ts)$ for any $t\in \mathbb{G}_m$. Taking quotients by the scaling $\mathbb{G}_m$-actions and restricting over $\pgrs$, we obtain a family of regular semisimple Hessenberg varieties 
\[ X_{\PP(H)}^{\rs}:=G\times^B\PP(H)|_{\pgrs}\lra \pgrs\]
over $\pgrs$, whose fiber over $[s]\in \pgrs$ is $X_H(s)$. 
This is a flat family of relative dimension $\dim ~H/\fb$, by Proposition~\ref{77}.

\medskip

We are primarily concerned with Hessenberg spaces that are obtained as (the preimages of) the tangent cones of $X_w$ at the identity.

\begin{definition}\label{def:Hess.cone}
Let $w \in W$. Denote by $C_w$ the tangent cone of $X_w$ at the point $eB$ corresponding to the identity. By the canonical inclusion $X_w\subset \cB$, it is naturally a closed subscheme in $T_{\cB,eB}\cong \fg/\fb$.
Define 
\[H_w:=\pi_\fb^{-1}(C_w)\] to be the preimage of $C_w$ under the natural projection $\pi_\fb:\fg\to \fg/\fb$. 

This defines a set-theoretic map
\beq\label{eq:corres}\;W\lra \{\text{Hessenberg subspaces}\}, \quad w\longmapsto H_w=\pi_\fb^{-1}(C_w),\eeq
whose restriction to the set of smooth elements of $W$ is the map \eqref{e.sm}.
\end{definition}  

Note that $H_w$ is $B$-invariant, since $C_w$ is. 
When $w$ is smooth, $C_w$ is the tangent space of $X_w$ at $eB$, and admits an explicit description as follows. \begin{lemma}\label{l:H_w.smooth}
	When $w$ is smooth, $H_w$ is a linear Hessenberg space \[H_w=\fb\oplus \Big(\bigoplus_{\a \in M_w} \fg_{-\a}\Big),\] where we write $M_w:=\{\alpha \in \Phi^+ \mid s_{\alpha} \leq w \}$. In particular, $\dim\,H_w/\fb=\ell(w)$. \end{lemma}
  \begin{proof} 

For arbitrary $w \in W$, the negative root spaces $\mathfrak g_{-\alpha}$ with $\alpha \in M_w$ correspond to $T$-stable curves in $X_w$ passing through $eB$ (\cite[Theorem~F(2)]{carrell}). Thus $M_w$ has at least $\ell(w)$ elements (\cite[Lemma in Section 2]{carrell}).    

The linear span of the tangent cone of $X_w$ at $eB$ is the $B$-module span of the negative root spaces $\fg_{-\alpha}$ with $\alpha \in M_w$ 
(See \cite[Section~4.4]{Po} or \cite[Section~5.6]{BL}). 
    When $X_w$ is smooth, the tangent cone at $eB$ is a linear space of dimension $\ell(w)$ and thus is linearly spanned by $\mathfrak g_{-\alpha}$ with $\alpha \in M_w$. 
\end{proof}

In general, when $w$ is singular, the tangent cone $C_w$ need not be a linear subspace,
and consequently $H_w$ may fail to be a linear subspace of $\fg$.
\begin{example}[{cf.~\cite[Section~1]{fuchs.et.al}}]
	Let $G=\GL_n$ and $W=S_n$. 
	
	When $n=4$ and $w=4231\in S_4$ in one-line notation, $w$ is singular and the tangent cone $C_w$ is isomorphic to a quadric cone 
	\[\{x_1x_2=x_3x_4\}\subset\A^6\] with an isolated singularity at the origin. In particular, $C_w$ is not linear.
	
	When $n=5$ and $w=45231\in S_5$, 
	$w$ is singular and 
	$C_w$ is reducible: 
	\[C_w=C_{w_1}\cup C_{w_2},\]
	where $w_1=35421$ and $w_2=43521$ are smooth. In particular, $H_w=H_{w_1}\cup H_{w_2}$ is the union of two linear Hessenberg spaces. Hence
	\[X_{H_w}(s)=X_{H_{w_1}}(s)\cup X_{H_{w_2}}(s)\]
	is the union of two smooth Hessenberg varieties, glued along their intersection 
	$X_{H_{w_3}}(s)$, where $w_3=34521$.
\end{example}

The map \eqref{e.sm}, which is the restriction of \eqref{eq:corres} to smooth elements $w$ and linear Hessenberg spaces $H$, is not surjective in general.

When $G$ is neither of type $A$ nor $C$, there is no smooth Schubert variety of codimension one \cite[Proposition 7.10]{Kumar96}, hence the map \eqref{e.sm} cannot be surjective. When $G$ is of type $C$, there is a linear Hessenberg space of higher codimension which is not in the image of the map \eqref{e.sm}; see Example~\ref{ex.c3} below.   It would be interesting to determine the image of the map \eqref{e.sm}.   

\begin{example} \label{ex.c3}  
   When $G$ is of type $C_3$, then there is no smooth $w \in W$ with $H=H_w$ for the linear Hessenberg space $H=\mathfrak b \oplus \left( \oplus_{\alpha \in M}\mathfrak g_{-\alpha}\right)$, 
   where $M$ is one of the following sets: \begin{eqnarray*} 
&&M_1:=\{\alpha_1, \alpha_2, \alpha_3, \alpha_1+\alpha_2, \alpha_2+ \alpha_3 \},    \\ 
&&M_2:=\{\alpha_1, \alpha_2, \alpha_3, \alpha_1+\alpha_2, \alpha_2+ \alpha_3, \alpha_1 +\alpha_2 + \alpha_3 \}.  
   \end{eqnarray*}

Recall that for smooth $w$, $H_w$ is given by $\mathfrak b \oplus \left( \oplus_{s_{\alpha} \leq w} \mathfrak g_{-\alpha}\right)$ (Lemma~\ref{l:H_w.smooth}).  We use the following property: If $s_{\alpha} \leq w$, then  any reduced expression of $w $ has a subword  that is a reduced expression for $s_{\alpha}$ (Corollary 3.2.3 of~\cite{BjBr}).

(1) Let $H$ be the linear Hessenberg space $\mathfrak b \oplus \left( \oplus_{\alpha \in M_1} \mathfrak g_{-\alpha}\right)$. Then $H$ has dimension 5. We list up all elements of length $5 $ in $W$:
\begin{eqnarray*}
 [23121], [12321], [32321], [13231],[12312], [32312],   [13232],[21323]. 
\end{eqnarray*}
Here, we use the notation $[abcde]$ for $s_as_bs_cs_ds_e$.

 Note that $\alpha_1+ \alpha_2 = s_2(\alpha_1)$ and $\alpha_2+ \alpha_3 =s_3(\alpha_2)$. If $s_{\alpha_1+ \alpha_2} \leq w$ and $s_{\alpha_2+ \alpha_3}\leq w$, then $w$ has both  $s_2s_1s_2$ (or $s_1s_2s_1$) and $s_3s_2s_3$ as subexpressions. Such $w$ is one of  
 \[[13231], [21323], [32312].\]

If $w=[13231]$, then $s_{\alpha_1+\alpha_2+\alpha_3}= s_1s_3s_2s_3s_1 \leq w$ but  $\alpha_1+\alpha_2+\alpha_3\not \in H$. If $w=[21323]$ or $w=[32312]$, then $s_{2\alpha_2 +\alpha_3} =s_2s_3s_2\leq w$ but  $2\alpha_2+ \alpha_3 \not \in H$. Therefore, $H$ cannot appear as $H_w$ for any smooth $w$.

(2) Let $H$ be the linear Hessenberg space  $\mathfrak b \oplus \left( \oplus_{\alpha \in M_2} \mathfrak g_{-\alpha}\right)$.  Then $H$ has dimension 6. We list up all elements of length 6 in $W$:
   \[[123121],[132321], [213231], [132312], [232312], [121323], [321323].\] 
    Here, again $[abcdef]$ denotes $s_as_bs_cs_ds_es_f$. All these reduced expression have $s_2s_3s_2$ as a subword. However, $2\alpha_2+ \alpha_3 \not \in H$. Therefore, $H$ cannot appear as $H_w$ for any smooth $w$.
\end{example}

\smallskip
\subsection{GKM graphs}\label{ss:GKM}
Recall that a smooth variety with an action of torus $T$ is called a \emph{GKM variety} if there are only finitely many $T$-fixed points and one-dimensional $T$-orbits, and if it is equivariantly formal (for example, if its cohomology group vanishes in odd degrees).
In this case, these $T$-fixed points and $T$-invariant curves form a graph equipped with decorations on the edges given by the $T$-weights on the corresponding $T$-invariant curves.

One advantage of a GKM variety is that its cohomology is completely determined by its GKM graph. More precisely, its $T$-equivariant cohomology admits an explicit description, as a subring of the $T$-equivariant cohomology of its fixed points, in terms of its GKM graph. Its singular cohomology is then obtained from this by taking quotient by the ideal generated by the equivariant parameters. We refer the reader to \cite{GKM} for the general theory.  

\smallskip

When $w\in W$ is smooth, both  
$\Y$ and $X_{H_w}(s)$ are smooth GKM varieties. Indeed, with respect to the natural action of the maximal torus $T=C_G(\s)$ or $C_G(s)$ centralizing $\s$ or $s$, these varieties have only finitely many fixed points and invariant curves since $\cB$ has, and their cohomology groups vanish in odd degrees, due to the Bia\l{}ynicki-Birula decomposition associated to an action of a generic one-dimensional subtorus in $T$.

Moreover, the GKM graphs for $X_{H_w}(s)$ are well-known (cf.~\cite[Section~8.2]{AHMMS}). Using this, Tymoczko \cite{tymoczko} defined the dot action of $W$ on the cohomology of $X_{H_w}(s)$, that gives rise to interesting graded representations of $W$. 

For example, in type $A$, it was conjectured by Shareshian and Wachs \cite{shareshian-wachs}, and proved by Brosnan and Chow \cite{brosnan-chow} and independently by Guay-Paquet \cite{guay-paquet}, that these representations correspond to the chromatic quasisymmetric functions (\cite{stanley-xg,shareshian-wachs}) of the indifference graphs of unit interval orders. 

On the other hand, the dot action of $W$ is well-defined on $H^*(\Y)$ for $w$ smooth as well, or more generally on the intersection cohomology $IH^*(\Y$) for any $w$. This is defined in \cite{AN} via the monodromy action of $\pi_1(\Grs)$, motivated by the fact in the case of Hessenberg variety, the dot action is equivalent to the monodromy action of $\pi_1(\grs)$, as proved in \cite{brosnan-chow}.

\smallskip
In this subsection, we prove that when $w$ is smooth, the GKM graph of $\Y$ is the same as that of $X_{H_w}(s)$, for any choice of regular semisimple $\s \in G$ and $s\in\fg$. In particular, the corresponding graded $W$-representations $H^*(\Y)$ and $H^*(X_{H_w}(s))$ are the same.  

\smallskip
This was already well-known when $G=\GL_n$. Thus, we review this first.
A \emph{Hessenberg function} is a function $h:[n] \rightarrow [n]$ such that $h(i)\geq i$ and $h(i)\leq h(j)$ for $i<j$, where $[n]:=\{1,\cdots, n\}$.
Note that these functions have a bijective correspondence with linear Hessenberg spaces in $\mathfrak{gl}_n$, via
\[h \quad \longleftrightarrow \quad  H_h:=\fb\oplus\Big(\bigoplus_{\a \in M_h}\fg_{-\a}\Big),\]
where we define 
\[M_h:=\{ \alpha_{i,j} \in \Phi^+: j \leq h(i) \text{ for all }i\in[n]  \}\] 
and $\alpha_{i,j}:=\alpha_i + \alpha_{i+1} + \dots + \alpha_{j-1}$ for $i<j$. Hence, we write $X_h(s):=X_{H_h}(s)$.

This bijection extends to the set of the \emph{codominant} permutations in $S_n$, which are by definition the maximal elements $w_h$ in $S_n$ satisfying 
\beq \label{eq:codom.def} w_h(i) \leq h(i)~ \text{ for all }i \in [n]\eeq
in the lexicographic order for some Hessenberg functions $h:[n]\to[n]$ (Definition in Section~3 of \cite{haiman}).
These are precisely  312-avoiding permutations.

\begin{proposition} \label{p normal form in type A}
For a Hessenberg function $h:[n] \rightarrow [n]$, there is a unique maximal element $w_h$ satisfying \eqref{eq:codom.def}.
Conversely, every codominant element $w$ in $S_n$ determines $h$ with $w=w_h$. 
In this case, $w_h$ has a reduced expression
\[w_h:= (s_{h(1)-1} \cdots s_2s_1)(s_{h(2)-1} \cdots  s_2)\cdots
(  s_{n-1}),\]
where $(s_{h(i)-1}\cdots s_i)$ is set to be an empty word when $h(i)=i$
\end{proposition}
\begin{proof}
	Left to the reader. For the second part, see \cite[Proposition~3.1]{haiman}.
\end{proof}

Based on this, one can see 
that smooth regular semisimple  
Lusztig varieties coincide with Hessenberg varieties in type $A$, as noted in \cite{AN2}. 
Identify $\GL_n\subset \mathfrak{gl}_n$ with the spaces of (invertible) $n\times n$ matrices. Then, $\GL_n^{\mathrm{rs}}\subset \mathfrak{gl}_n^{\mathrm{rs}}$.
\begin{proposition}
[\cite{AN2}]  \label{p codominant}
For $\s\in \GL_n$ regular semisimple,
$Y_{w_h}(\s)=X_h(\s)$.
\end{proposition} 

More interestingly, it is observed in \cite{AN2} 
that there exists a natural map
\beq\label{eq:corres.A} \{\text{smooth elements in }S_n\}\lra \{\text{Hessenberg functions }h:[n]\to [n]\}\eeq
that restricts to a bijection on the set of codominant elements, sending $w_h$ to $h$.
In particular, \eqref{eq:corres.A} is surjective. 
Moreover, if two elements have the same image, then the associated Lusztig varieties have the same GKM graphs.
\begin{theorem}[Theorem~1.6 in \cite{AN2}] \label{thm:AN}
Let $\s\in \GL_n$ be regular semisimple.
Let $w$ and $w'\in S_n$ be smooth
and have the same image under \eqref{eq:corres.A}.  Then, $\Y$ and $Y_{w'}(\s)$ have the same GKM graphs. In particular,
\[H^*(\Y)\cong H^*(Y_{w'}(\s))\] 
as graded $S_n$-representations.
\end{theorem}
This confirms one of the Haiman's conjectures on the Hecke characters of Kazhdan-Lusztig basis elements (\cite[Conjecture~3.1]{haiman}) when $w$ is smooth, while it is disproved for singular $w$, as shown by a counterexample in \cite[Theorem~1.8]{AN2}. Note also that the smooth case was originally proved in \cite{CHSS}.

\smallskip

In the remainder of this subsection, we will formulate and prove a generalization of Theorem~\ref{thm:AN} to arbitrary Lie types.
To this end, we first extend
the map in \eqref{eq:corres.A}. The following observations in type $A$ suggest that 
such an extension should be formulated in terms of $H_w$.

\begin{proposition} \label{p.w and H A} 
Let $h:[n]\to [n]$ be a Hessenberg function and let $w_h\in S_n$ be the corresponding codominant permutation. Then, 
$H_h=H_{w_h}$.
\end{proposition} 
 
\begin{proof}  
Each $\alpha_{i,j}$ can be written as $\alpha_{i,j} = s_{j-1}\dots s_{i+2}s_{i+1}(\alpha_i)$ for all $i<j$. Thus, for any $\alpha \in M_h$, $s_{\alpha}$ has a reduced expression which is a subexpression of $w=w_h$ given in Proposition \ref{p normal form in type A}, and  we get $s_{\alpha} \leq w$. Therefore, $M_h$ is contained in $M_w$. Since $w$ is smooth, $M_h$ is equal to $M_w$, so $H_h=H_{w_h}$.
\end{proof} 
\begin{corollary}\label{cor:corres.A}
	Let $w\in S_n$ be smooth and let $h$ be its associated Hessenberg function. Then, $H_w=H_h$. In particular, \eqref{eq:corres.A} 
	agrees with \eqref{e.sm}.
\end{corollary}
\begin{proof}
	Note that $\Y$ and $Y_{w_h}(\s)$ have the same tangent spaces at $eB$, viewed as a subspace in $\fg/\fb$, since they have the same GKM graphs by Theorem~\ref{thm:AN}. 
	Hence, $H_w=H_{w_h}$ is equal to $H_h$ by Proposition~\ref{p.w and H A}.
\end{proof}
\smallskip

From now on, we let $G$ be of arbitrary type. 

\smallskip
The following extends 
Theorem~\ref{thm:AN} by Abreu and Nigro to arbitrary type. 
Before stating the theorem, we make one remark: the automorphism $\cB\xrightarrow{\cong}\cB$ given by the left multiplication by $g\in G$   
induces automorphisms 
\beq\label{eq:conjugate}\Y \cong Y_{w}(g.\s) \and X_H(s)\cong X_H(g.s)\eeq
where we denote $g.\s:=g\s g^{-1}$. So when computing the GKM graphs or the cohomology of $\Y$ or $X_H(s)$, we may assume $\s \in T$ and $s\in \ft$.

Let $T^{\rs}:=T\cap \Grs$ and $\ft^{\rs}:=\ft\cap \grs$.
\begin{theorem} \label{thm GKM graph}
Let $w\in W$ be smooth. Let $\s \in T^{\rs}$ and $s\in \ft^{\rs}$.
Then, $\Y$ and $X_{H_w}(s)$ have the same GKM graphs. 
In particular, 
if two smooth elements $w$ and $w'$ in $W$ have the same image $H=H_w=H_{w'}$ under \eqref{eq:corres},
then $\Y$ and $Y_{w'}(\s)$ have the same GKM graphs. Consequently,
\[H^*(\Y)\cong H^*(X_{H}(s))\cong H^*(Y_{w'}(\s))\]
as graded $W$-representations
\end{theorem}

\begin{proof}
	This follows from the proof of \cite[Proposition 8.2]{AHMMS}, where the GKM graph of $X_H(s)$ is computed. Since its vertices are equal to those for $\cB$, it suffices to know which edges are contained in $X_H(s)$. We recall this first. 

For each root $\a$, let $U_\a \subset G$ be the \emph{root subgroup} associated to $\a$. It is the unique one-dimensional connected subgroup of $G$ normalized by $T$ with $\mathrm{Lie}(U_\a)=\fg_\a$ (\cite[Proposition~8.1.1]{Sp}).
Every $T$-invariant curve in $\cB$ is the closure of  $U_\a v$ with $v\in W$ and $\a\in \Phi^+$ such that $-v^{-1}\a\in \Phi^+$ (\cite{carrell} or \cite[Proposition~4.6]{tymoczko-schubert}). Here, by abuse of notation, we also denote by $v$ any lift of $v\in W$ in $\cB$.
In \cite[(8.4)]{AHMMS}, it is shown that 
\[(U_\a v)^{-1}.s  =\fg_{v^{-1}\alpha}\]
in $\fg/\fb$, where $\fg_{\beta}\subset \fg$ is the root space of a root $\beta$. Therefore, $U_\a v\subset X_H(s)$ if and only if $\fg_{v^{-1}\a}\subset H$. This determines the GKM graph of $X_H(s)$.

Similarly, since $U_\a v=vU_{v^{-1}\a}$ and $U_{v^{-1}\a}^{-1}=U_{v^{-1}\a}$, we have
\[(U_\a v)^{-1}.\s=(vU_{v^{-1}\a})^{-1}.\s=U_{v^{-1}\a}.(v^{-1}.\s) ~\subset~ U_{v^{-1}\a}B\]
in $G$,
where the last inclusion holds since $U_{v^{-1}\a}$ is normalized by $T$. Indeed,
$U_{v^{-1}\a}.T \subset U_{v^{-1}\a}T U_{v^{-1}\a}=U_{v^{-1}\a}^2T \subset U_{v^{-1}\a}B$.
Consequently, we have
\beq\label{eq:Tinvar.curve}(U_\a v)^{-1}.\s~\subset~ U_{v^{-1}\a}e\eeq
in $G/B=\cB$.
Since $(U_\a v)^{-1}.\s$ is not a point while $U_{v^{-1}\a}e$ is a $T$-invariant (affine) curve in $\cB$, the inclusion \eqref{eq:Tinvar.curve} induces the equality of the closures  
\[\overline{(U_\a v)^{-1}.\s}= \overline{ U_{v^{-1}\a}e}\] in $\cB$. 
Moreover, its tangent space at $e\in \cB$ is $\fg_{v^{-1}\a}\subset \fg/\fb$.

This shows that $\overline{U_\a v}$ (as a $T$-invariant curve in $\cB$) is contained in $\Y$ if and only if the $T$-invariant curve $\overline{U_{v^{-1}\a}e}$ passing through $e$ is contained in $X_w$. Since $e\in X_w$ is a smooth point, this is equivalent to $\fg_{v^{-1}\a}\subset T_{X_w,eB}=H_w/\fb$, hence equivalent to $\overline{U_\a v}\subset X_{H_w}(s)$. This shows that $\Y$ and $X_{H_w}(s)$ have the same $T$-invariant curves, hence the same GKM graphs.

The remaining assertions are now immediate from this. 
\end{proof}

\begin{proof}[Proof of Corollary~\ref{cor:Wrep}]
	Now immediate from Theorem~\ref{thm GKM graph} and \eqref{eq:conjugate}.
\end{proof}

We will lift this relationship between $\Y$ and $X_{H_w}(s)$ to the level of varieties by constructing a family of these varieties where one degenerates to the other, in the next subsection. It will lead to a result on their diffeomorphism types that specializes to a proof of Conjecture~\ref{conj} in type $A$.

\medskip

\subsection{Hessenberg varieties as degenerations of  Lusztig varieties}  
\label{ss:degeneration}
Consider the universal families $Y_w^{\rs}\to \Grs$ of regular semisimple Lusztig varieties and $X_{\PP(H_w)}^{\rs}\to \pgrs$ of regular semisimple Hessenberg varieties.

Let $\tG\to G$ denote the blowup at the identity $e$. 
Let $\widehat G$ be the complement of the proper transform of $G-\Grs$.
Define $\widehat E$ to be the intersection of the exceptional divisor with $\widehat G$.
By construction, $\Grs$ is an open subset of $\widehat G$, with $\widehat E\cong \pgrs$ as its complement. Moreover, $\widehat G$ is smooth and connected.

In this section, we prove the following. 
\begin{theorem}\label{thm:univ.fam}
	There exists a flat projective morphism $\widehat Y_w \to \widehat G$ of relative dimension $\ell(w)$ over 
    $\widehat G$ which satisfies the following. 
	\begin{enumerate}
		\item The family restricts to $Y_w^{\rs} \to \Grs$ over $\Grs$.
		\item The family restricts to $X_{\PP(H_w)}^{\rs}\to \pgrs$ over $\pgrs$. 
	\end{enumerate}
	Moreover, the family is smooth if $w$ is smooth.
	In short,
	\beq\label{eq:univ.fam.diagram}\begin{tikzcd}
			Y_w^{\rs}\arrow[r,hook]\arrow[d]&\widehat Y_w \arrow[d]&X_{\PP(H_w)}^{\rs}
			\arrow[l,hook']\arrow[d]\\ 
			\Grs\arrow[r,hook]& \widehat G&\pgrs \arrow[l,hook']
		\end{tikzcd}\eeq
	where the left and right squares are Cartesian and $\widehat G=\Grs\sqcup \pgrs$.
\end{theorem}

The rest of this subsection is devoted to constructing the 
morphism $\widehat Y_w\to
\widehat G$ and proving Theorem~\ref{thm:univ.fam}. The idea is to apply the
deformation-to-normal-cone type argument in \cite{fulton-intersection} to the
twisted diagonals $\iota_\s$ for varying $\s\in G$.

We begin by recalling the construction of the universal family $Y_w\to G$ of Lusztig varieties over $G$, given as the following fiber product:
\beq\label{72}
	\begin{tikzcd}
		Y_w\arrow[r,hook]\arrow[d,hook']&G\times \fX_w\arrow[d,hook']\\ G\times \cB\arrow[r,hook]&G\times \cB\times \cB
	\end{tikzcd} 
\eeq
Here, the lower horizontal map sends $(x,gB)$ to $(x,gB, xgB)$. This induces a family of diagrams over $G$ via the first projections, whose fiber over $\s\in G$ is the diagram \eqref{3}. This $Y_w\to G$ is the universal family of Lusztig varieties.

Applying base change to \eqref{72} along $\widehat G \to G$, we obtain
\beq\label{71}\begin{tikzcd}
			&\widehat G\times \fX_w\arrow[d,hook']\\\widehat G\times \cB\arrow[r,hook]&\widehat G\times \cB\times \cB.
		\end{tikzcd}\eeq
Moreover, all three spaces in \eqref{71}, as subvarieties of $\widehat G\times \cB\times \cB$, contain $\widehat E\times \Delta(\cB)$, since the corresponding three spaces in \eqref{72} contain	$e\times \Delta(\cB)$ and $\widehat E$ is the preimage of $e$ in $\widehat G$.
Blowing up along $\widehat E\times \Delta(\cB)$ and taking the fiber product, we obtain 
\beq\label{70}\begin{tikzcd}
			\widehat Y_w \arrow[r,hook]\arrow[d,hook']&\Bl_{\widehat E\times \Delta(\cB)}(\widehat G\times \fX_w)\arrow[d,hook']\\  \Bl_{\widehat E\times \cB}(\widehat G\times \cB )\arrow[r,hook]&\Bl_{\widehat E\times \Delta(\cB)}(\widehat G\times \cB\times \cB)
		\end{tikzcd}\eeq
where $\Bl_{\widehat E\times \cB}(\widehat G\times \cB )\cong \widehat G\times \cB$.
Thus, by the first projection, we obtain
\beq\label{eq:simul.fam}\widehat Y_w\lra \widehat G\eeq
each of whose fibers is a closed subscheme in $\cB$. 
Since the diagram~\eqref{70} coincides with \eqref{72} when restricted over $\Grs$, the map \eqref{eq:simul.fam} restricts to $Y_w^{\rs}\to \Grs$ over $\Grs$.
Consequently, we obtain the left square in \eqref{eq:univ.fam.diagram},
leaving only (3) in Theorem~\ref{thm:univ.fam} to be proved.  
Thus, we now verify (3).
\begin{lemma}
	The map \eqref{eq:simul.fam} restricts to $X_{\PP(H_w)}^{\rs}\to \pgrs\cong \widehat E$ over $\widehat E$.
\end{lemma}
\begin{proof}
	The fibers of $\Bl_{\widehat E \times \Delta(\cB)}(\widehat G\times \cB\times \cB)\to \widehat G$ over points in $\widehat E$ are
	\beq\label{73}\Bl_{\Delta(\cB)}(\cB\times \cB)~\cup~ \PP_{\Delta(\cB)}\left(\cO_{\Delta(\cB)}\oplus N_{\Delta(\cB)/\cB\times \cB}\right),\eeq
	glued along $\PP_{\Delta(\cB)}(N_{\Delta(\cB)/\cB\times \cB)})$,
	where $\cO_{\Delta(\cB)}$ is the restriction of $N_{\widehat E/\widehat G}$.
	Similarly, the fibers of $\Bl_{\widehat E\times \Delta(\cB)}(\widehat G\times \fX_w) \to \widehat G$ over points in 
	$\widehat E$ are 
	\beq\label{74}\Bl_{\Delta(\cB)}\fX_w~\cup~ \PP_{\Delta(\cB)}\left(\cO_{\Delta(\cB)} \oplus C_{\Delta(\cB)/\fX_w}\right)\eeq
	glued along $\PP_{\Delta(\cB)}(C_{\Delta(\cB)/\fX_w})$, where $C_{\Delta(\cB)/\fX_w}$ denotes the normal cone of $\Delta(\cB)$ in $\fX_w$ so that $C_{\Delta(\cB)/\fX_w}=N_{\Delta(\cB)/\fX_w}$ if $w$ is smooth. 
	Note that \eqref{73} and \eqref{74} contain $N_{\Delta(\cB)/\cB\times \cB}\cong T_\cB$ and $C_{\Delta(\cB)/\fX_w}\cong G\times^BC_w$ respectively.
	
	On the other hand, the preimage of $\widehat E$ under  $\Bl_{\widehat E\times \cB}(\widehat G\times \cB )\cong \widehat G\times \cB \to \widehat G$
	is canonically isomorphic to $\PP_{\widehat E}(N_{\widehat E/\widehat G})\times \cB\cong \PP_{\widehat E}(\cO_{\widehat E}(-1))\times \cB$. 
	Under the lower horizontal map in \eqref{70}, this maps into 
	\[\PP_{\widehat E \times \cB}\left(pr_1^*\cO_{\widehat E}(-1)\oplus pr_2^*N_{\Delta(\cB)/\cB\times \cB}\right) \cong \PP_{\widehat E \times \cB}\left(pr_1^*\cO_{\widehat E}(-1)\oplus pr_2^*T_\cB\right)\]
	which is the union of the right component in \eqref{73} over points in $\widehat E$, for the projections $\widehat E\xleftarrow{pr_1}\widehat E\times \cB \xrightarrow{pr_2} \cB\cong \Delta(\cB)$. Furthermore, one can easily check, for example using one-parameter families of Lusztig varieties, that this map sends $([s],gB)$ to $[s:[g,g^{-1}s]]$, a point lying over $([s],gB)\in \widehat E\times \cB$.
	
	In particular, it further factors through as
	\[\widehat E\times \cB ~\hooklongrightarrow~  pr_2^*T_\cB. \]
	Here, $pr_2^*T_\cB$ is the locus where the first entry is nonzero, that is, it is the union of the complement of the left component in \eqref{73}, over points in $\widehat E$.
	
	The inclusion map restricts to the map 
	\[\cB ~\hooklongrightarrow~  T_\cB\] 
	that sends $([s],gB)$ to $[g,g^{-1}.s]$ over $[s]\in \widehat E\cong \pgrs$, which is precisely the lower horizontal map in the fiber diagram \eqref{75}.
	Moreover, the preimage of $N_{\Delta(\cB)/\cB\times \cB}$ in \eqref{74} is $C_{\Delta(\cB)/\fX_w}\cong G\times^BH_w/\fb$. Therefore, the fiber of \eqref{eq:simul.fam} over $[s]$ 
	is the fiber product \eqref{75} with $H=H_w$. 
	Hence, it is $X_{H_w}(s)$.
\end{proof}

This proves (3) in Theorem~\ref{thm:univ.fam}. 
Since $\widehat G$ is smooth and  all the fibers have the same dimension (resp. and they are smooth), \eqref{eq:simul.fam} is flat (resp. smooth, when $w$ is smooth). This completes the proof of Theorem~\ref{thm:univ.fam}. \hfill $\square$

\begin{remark}\label{rem:trivial.bundle}
By construction, the family $\widehat Y_w$ is a closed subscheme of the trivial
$\cB$-bundle $\widehat G\times \cB$; see \eqref{70}.
\end{remark}

The following strengthening of Theorem~\ref{thm GKM graph} is now immediate from Theorem~\ref{thm:univ.fam} and the Ehresmann theorem.

\begin{corollary}\label{cor:diffeo}
	Let $w$ and $w'\in W$ be smooth elements with $H_w=H_{w'}$. Then, $\Y$ and $Y_{w'}(\s')$ are diffeomorphic for any $\s$ and $\s'\in \Grs$.
\end{corollary}
\begin{proof}
	Applying the Ehresmann theorem to $\widehat Y_w\to \widehat G$ and $\widehat Y_{w'}\to \widehat G$,   
	$\Y\cong X_H(s)\cong Y_{w'}(\s')$ are diffeomorphic for any $s\in \grs$ and $H:=H_w=H_{w'}$.
\end{proof}
Restricting to the case of type $A$, this in particular proves Conjecture~\ref{conj}.

\subsection{Cohomology of line bundles revisited}\label{ss:cohomology.revisited}
Combining Theorems~\ref{thm:vanishing} 
 and \ref{thm:univ.fam}, we can compute the $C_G(\s)$-representation
 \[V_w(\lambda):=H^0(\Y,L_\lambda)\]
 for dominant $\lambda$ and smooth $w$.
 Note that the virtual $C_G(\s)$-representation
\[\chi^{C_G(\s)}(\cB,L_\mu):=\sum_{i\geq0}(-1)^iH^i(\cB,L_\mu)\]
viewed as a class in the Grothendieck group of $C_G(\s)$-representations
is computable for any integral weight $\mu$ by the Borel-Weil-Bott theorem (\cite[p.~103]{Ak}). Forgetting the representation structure recovers 
the Euler characteristic $\chi(\cB,L_\mu)=\sum_{i\geq0}(-1)^i\dim\,H^i(\cB,L_\mu)$.
\begin{proposition}\label{prop:rep}
	Let $w\in W$ be smooth and $\s\in \Grs$. Let $\lambda$ be a dominant
integral weight. Then
	\[V_w(\lambda)=\sum_{I\subset \Phi^+-M_w}(-1)^{\lvert I \rvert}\chi^{C_G(\s)}(\cB, L_{\lambda -\langle I\rangle})\]
	where $M_w$ is as in Lemma~\ref{l:H_w.smooth} and $\langle I\rangle:= \sum_{\a \in I}\a$ denotes the sum of all the  roots in $I$.
	In particular, 
	\[\dim\, V_w(\lambda)=\sum_{I\subset \Phi^+-M_w}(-1)^{\lvert I \rvert}\chi(\cB, L_{\lambda -\langle I\rangle}).\]
\end{proposition}
\begin{proof}
	By Theorem~\ref{thm:vanishing} or \ref{thm:vanishing.Frob}, the higher cohomology of $L_\lambda$ on $\Y$ vanishes for dominant $\lambda$. Hence $V_w(\lambda)=\chi^{C_G(\s)}(\Y,L_\lambda)$. Moreover, by Theorem~\ref{thm:univ.fam} (cf.~Remark~\ref{rem:trivial.bundle}), there exists a $C_G(\s)$-equivariant flat degeneration from $\Y$ to $X_{H_w}(s)$, for any $s\in\grs$ with $C_G(s)=C_G(\s)$, along which the line bundle $L_\lambda$ extends.
	It follows that 
	\[V_w(\lambda)=H^0(\Y,L_\lambda)=\chi^{C_G(\s)}(\Y,L_\lambda)=\chi^{C_G(\s)}(X_{H_w}(s),L_\lambda).\]

	Since $X_{H_w}(s)$ is the zero locus of a regular section of the homogeneous vector bundle $E:=G\times^B(\fg/H_w)\to\cB$, the dual map $E^\vee\to\cO_\cB$ gives rise to a Koszul resolution of $\cO_{X_{H_w}(s)}$, and hence of $L_\lambda|_{X_{H_w}(s)}$.
	By Lemma~\ref{l:H_w.smooth}, the bundle $E$ is an iterated extension of line bundles $L_\a$ with $\a\in\Phi^+-M_w$.
	
	Applying the Leray spectral sequence to this Koszul resolution yields the desired formula, since all constructions involved are $C_G(\s)$-equivariant.
\end{proof}

An immediate consequence of Proposition~\ref{prop:rep} is Corollary~\ref{cor:weight.multi}. \begin{proof}[Proof of Corollary~\ref{cor:weight.multi}]
	By Proposition~\ref{prop:rep}, the assertion reduces to the corresponding statement for the virtual representation $\chi^{C_G(\s)}(\cB,L_\mu)$, for (not necessarily dominant) integral weights $\mu$. This follows immediately from the Borel-Weil-Bott theorem and the Weyl character formula.
\end{proof}

As immediate corollaries of Theorems~\ref{thm:vanishing} and~\ref{thm:vanishing.Frob} and Proposition~\ref{p codominant}, we obtain the following results for Hessenberg varieties in type $A$.

\begin{corollary}\label{cor:vanishing.hess}
	Let $s\in \mathfrak{gl}_n$ be regular semisimple. Let $h:[n]\to[n]$ be a Hessenberg function. Let $L$ be a line bundle on $X_h(s)$.
	\begin{enumerate}
		\item If $L$ is nef, then $H^i(X_h(s),L)=0$ for $i>0$.
		\item If $L$ is ample and $h':[n]\to[n]$ is another Hessenberg function with $h'(i)\leq h(i)$ for all $i\in [n]$, 
		then the restriction map $H^0(X_h(s),L)\to H^0(X_{h'}(s),L)$ is surjective.
	\end{enumerate}
\end{corollary}
\begin{remark}\label{rem:AFZ}
	In \cite{AFZ}, Abe, Fujita and Zeng classified weak Fano regular semisimple Hessenberg varieties in type $A$ (\cite[Theorem~B]{AFZ}). As a result, the vanishing in (1) holds for this class of Hessenberg varieties (over characteristic zero) by the Kawamata-Viehweg vanishing theorem (\cite[Corollary~D]{AFZ}). Our Corollary~\ref{cor:vanishing.hess}(1) extends this vanishing result to arbitrary $h$.
\end{remark}

\bibliographystyle{plain}
\def\cprime{$'$}

\end{document}